\documentclass[12pt]{article}

\usepackage[T1]{fontenc}
\usepackage[english]{babel}

\usepackage[a4paper, margin=3.2cm]{geometry}
\usepackage{amsmath, amssymb, amsthm}
\usepackage{bbm}
\usepackage[hidelinks]{hyperref}

\newtheorem{thm}{Theorem}
\newtheorem{lem}{Lemma}
\newtheorem{rem}{Remark}
\newtheorem{prop}{Proposition}
\newtheorem*{AH}{Assumption (H)}
\newtheorem*{AP}{Assumption (P)}

\newcommand{\R}{\mathbb{R}}
\newcommand{\C}{\mathbb{C}}
\newcommand{\N}{\mathbb{N}}
\newcommand{\Z}{\mathbb{Z}}
\newcommand{\T}{\mathbb{T}}

\usepackage{graphicx}
\usepackage{xcolor}

\newcommand{\eps}{\varepsilon}
\def\rest{\hskip 1pt{\hbox to 10.8pt{\hfill
\vrule height 7pt width 0.4pt depth 0pt\hbox{\vrule height 0.4pt
width 7.6pt depth 0pt}\hfill}}}

\title{On partial diffusion and mixing without hypoellipticity}
\author{Xu'an Dou, Delphine Salort and Didier Smets}
\date{}

\begin{document}
\maketitle

\begin{abstract}
A simple Markov process is considered involving a diffusion in one direction and
a transport in a transverse direction. Quantitative mixing rate estimates are 
obtained with limited assumptions about the transport field, which might be highly 
irregular and/or highly degenerate, in particular quite far from satisfying an 
hypoellipticity type assumption.
\end{abstract} 

\section{Introduction}
We consider the linear partial differential equation 
\begin{equation}\label{eq:u}
\partial_t u - \partial_{xx}u + V(x)\partial_y u = 0 
\end{equation}
for the unknown $ u \equiv u(t,x,y) : \R_+ \times \T \times \T \to \R $, where $
\T = \R/\Z $ denotes the unit flat torus, and $ V : \T \to \R $ is a given
bounded Borel measurable function.
It reflects the interplay between a linear diffusion in one variable $x$, 
and a transport at velocity $V(x)$ in a transverse variable $y$. 

Our aim is to study the mixing properties induced by 
this coupling, and more precisely to derive sufficient conditions, 
in terms of $V$ and of the initial data, ensuring the exponential 
relaxation towards the unique invariant state with the same mass, 
here a constant state.

Equation~\eqref{eq:u} is arguably the simplest model for which the
problem makes sense and yet doesn't necessarily have an immediate answer. The 
principal part of the operator $L := -\partial_{xx} + V(x)\partial_y$ 
is only semi-definite, $L$ is not coercive in an $L^2$ sense since  
\begin{equation*}\label{eq:nop}
	(Lu, u) =  \int_\T\int_\T |\partial_x u|^2 \,dx dy, 
\end{equation*}
and there is no classical Poincaré inequality for the latter quantity as it vanishes
for functions that only depend on $y$. 

Note that one has the identity 
\begin{equation*}\label{eq:hor}
	\partial_t + L = X_0 - X_1^2, \qquad X_0 := \partial_t +
	V(x)\partial_y, \quad X_1 := \partial_x,    
\end{equation*}
in terms of the derivations $X_0$ and $X_1.$ In a celebrated paper \cite{Hor67}, 
H\"ormander proved that second order differential operators of the form 
$
X_0 - \sum_{i=1}^r X_i^2,
$
are hypoelliptic provided, at each point $m$ of a $\mathcal{C}^\infty$ manifold 
$M$, the iterated Lie brackets between the $\mathcal{C}^\infty$ vector 
fields $X_i$ ($0\leq i \leq r$) on $M$ span all of $T_mM$. That 
condition is also necessary when the vector fields are analytic, and in 
the general $\mathcal{C}^\infty$ case Amano \cite{Ama79} showed that it 
is necessary that the system
\begin{equation}\label{eq:controlsys}
	\dot q(s) = \sum_{i = 0}^r \xi_i(s) X_i(q(s)),
\end{equation}
with real valued controls $\xi_i \in \mathcal{C}^\infty$, is
controllable in every subdomain of $M.$\\
The isotropic and $L^2$
based subelliptic estimates of \cite{Hor67} were extended by Folland, 
Rothschild and Stein \cite{FolSte74, Fol75, RotSte76} into anisotropic 
and $L^p$ estimates for the fundamental solutions, and Jerison \cite{Jer86} 
later derived the Poincaré inequality, on balls of the so-called 
control distance associated to the vector fields $X_i.$\\
In the simple framework of \eqref{eq:u}, the commutators are given by
$$
C_0 := X_0 \quad\text{ and } \quad C_{k+1} := [X_1, C_{k}] = V^{(k)}(x)\partial_y,
\quad\text{ for all } k \geq 0
$$ 
(other forms of commutators identically vanish), and therefore 
hypoellipticity occurs when critical points of $V$ are at 
most finitely degenerate, and hence of finite number. 

\smallskip

Although not directly related through its definition, hypoellipticity itself 
is a sufficient condition for exponential relaxation towards 
equilibrium for \eqref{eq:u}. 
As a matter of fact, a stronger conclusion actually holds in such cases. 
Indeed, it was proved by Bedrossian and Coti Zelati 
\cite{BedCot17}, and D. Wei \cite{Wei21}, that in the context of 
Equation~\eqref{eq:u} with an additional coefficient $\nu > 0$ in front 
of the diffusion term, not only exponential relaxation holds (say, in
$L^2(\T^2)$), but 
its rate, for sufficiently large frequencies 
in $y$ and as $\nu \to 0$, is of magnitude $O(\nu^\frac{n}{n+2}),$ 
where $n$ is the highest degeneracy order of critical points of $V$. 
In particular, it is larger than $O(\nu)$, which would correspond to 
a purely diffusive situation. That phenomenon has been termed {\it enhanced} 
dissipation, in this and related contexts. 

\smallskip

On the other hand, hypoellipticity, which enforces a condition 
locally everywhere, is certainly not a necessary condition for 
exponential relaxation, which has a more 
global nature. The notion of global hypoellipticity (i.e. smooth 
second terms have smooth solutions), which is weaker than its 
local counterpart, is already closer to our goal. In  
\cite{FujOmo83}, Fujiwara and Omori exhibited  
globally hypoelliptic operators of the form $X_0^2 + X_1^2 = 
\partial_{xx} + V^2(x)\partial_{yy}$, where $V$ is smooth but 
constant on some intervals (we shall call such {\it plateaux}) and hence 
indefinitely degenerate. Because of the plateaux, they do not satisfy
the Fefferman-Phong condition 
\cite{FefPho83}, and therefore do not satisfy subelliptic
estimates either. The short proofs in \cite{FujOmo83} are 
based on Fourier decomposition in $y$ and spectral estimates in each mode (note 
the operators are symmetric in that case). Omori and Kobayashi \cite{OmoKob99}  
later conjectured that, for sums of squares, a sufficient condition for 
global hypoellipticity is the global controllability of \eqref{eq:controlsys}.  

\medskip

Note that an immediate necessary condition on $V$ for mixing 
in \eqref{eq:u} is that it should not be identically constant. 
Otherwise indeed, the  straightforward change of variable 
$y \mapsto y - Vt$ reduces it to a pure diffusion in $x$, which 
obviously does not induce any mixing in $y.$ As we shall see,  
that condition is also sufficient, provided the initial data is in 
$L^2(\T^2)$ and the relaxation is also understood in $L^2(\T^2)$.  
In the next paragraph though, we recall that Equation~\eqref{eq:u} 
has also a probabilistic interpretation as a Markov process. For that 
reason, or simply because \eqref{eq:u} preserves mass, we are  
specially interested in obtaining relaxation estimates that are valid in 
the context of $L^1$ or even Radon measure initial data, and we stress 
upfront that there will be in general no $L^1$ to $L^2$ smoothing 
effect under our assumptions.
In the sequel of this paper, we shall therefore be  
interested in obtaining quantitative relaxation estimates 
for Radon measure initial data and for large classes of $V$, including 
different non (globally) hypoelliptic regimes, such as when the 
velocity field $V$ is rough, has jumps, or exhibits plateaux.
Such situations appear in recent modeling of bacteria density 
(see \cite{JLT2025} and references therein, where the velocity field 
of the internal variable can be piecewise constant), and imply 
highly non-uniform  mixing across phase space. Plateaux, e.g., act as partial traps in 
which the mass dissipates only slowly and without local smoothing 
effects in $y$. 

\medskip

From a probabilistic perspective, Equation~\eqref{eq:u} is associated to the
stochastic Markov process $(\mathcal{X}_t, \mathcal{Y}_t)$ where
\begin{equation}\label{eq:sp}
\mathcal{X}_t := \mathcal{X}_0 + \sqrt{2} \mathcal{W}_t \mod 1, \qquad 
\mathcal{Y}_t = \mathcal{Y}_0 + \int_0^t V(\mathcal{X}_s) \, ds \mod 1,
\end{equation}
and where $\mathcal{W}_t$ refers to a standard Wiener process in $\R$ starting from 
the origin. If we denote by $\mu_t$ the law of $(\mathcal{X}_t, \mathcal{Y}_t)$, then for any 
$t \geq 0$ and any $\varphi \in \mathcal{C}^\infty([0,t], \T^2)$ we have, by Ito's formula, 
\begin{equation*}\label{eq:weaklaw}
\int_{\T^2} \varphi(t,\cdot,\cdot) d\mu_t = \int_{\T^2} \varphi(0, \cdot,\cdot) d\mu_0 +
\int_0^t\int_{\T^2} \left[ \partial_t + \partial_{xx} +
V(x)\partial_y\right]\varphi \, d\mu_s ds,   
\end{equation*}
which is a weak formulation for \eqref{eq:u}. The transition probabilities
\begin{equation}\label{eq:tp}
P_t(x,y,B) := \mathbb{P}\big( (\mathcal{X}_t,\mathcal{Y}_t) \in B\: | \: 
(\mathcal{X}_0, \mathcal{Y}_0) = (x,y) \big)
\end{equation}
are defined for $t \geq 0$, $(x,y) \in \T^2$, and $B \in \mathcal{B}(\T^2)$, 
the set of Borel subsets of $\T^2$. In PDE
contexts, these correspond to the so-called fundamental 
solutions or Green functions for the linear operator associated to \eqref{eq:u}, 
that is the solutions corresponding to Dirac $\delta_{x,y}$ initial data; but 
without further assumptions on $V$, they need
not be absolutely continuous (and a fortiori not smooth) with respect 
to the Lebesgue measure on $\T^2$. 
For fixed $t\geq 0$ and $B \in \mathcal{B}(\T^2)$, $(x,y) \mapsto P_t(x,y,B)$ is 
Borel measurable from $\T^2$ to $\R$, and for fixed $t\geq 0$ and $(x,y) \in \T^2$, 
$B \mapsto P_t(x,y,B)$ is a probability measure on $(\T^2,\mathcal{B}(\T^2))$.\\
The Markov property is reflected in the Chapman-Kolmogorov equations:
\begin{equation}\label{eq:CK}
P_{t+s}(x, y, B) = \int_{\T^2} P_t(x, y, dx'dy')P_s(x', y', B),
\end{equation}
for all $t,s\geq 0$ and $B \in \mathcal{B}(\T^2).$ The
transition operators $(S(t))_{t\geq 0}$ acting in the space $\mathcal{M}(\T^2)$ 
of Radon measures on $\T^2$ and defined by 
$$
\big(S(t) \mu\big)(B) := \int_{\T^2} P_y(x,y,B) \mu(dxdy) \qquad\forall B \in \mathcal{B}(\T^2)
$$
hence satisfy the semi-group property $S(t+s) = S(t)S(s)$, $\forall s,t \geq 0$. 
That semi-group is conservative: 
\begin{equation*}%\label{eq:conserv}
\int_{\T^2} S(t)\mu = \int_{\T^2} \mu \qquad \forall\, t\geq 0, \forall\, \mu \in
\mathcal{M}(\T^2),
\end{equation*}
and order preserving: 
\begin{equation*}%\label{eq:order}
S(t)\mu \geq S(t) \nu \qquad \forall\, t\geq 0, \forall\, \mu \geq \nu \in 
\mathcal{M}(\T^2).
\end{equation*}
From the assumption that $V$ is bounded, it follows that $(S(t))_{t\geq 0}$ is 
a weakly-$\star$ continuous semi-group on $\mathcal{M}(\T^2)$. Using the 
maximum principle for \eqref{eq:u}, it also follows that the restriction of 
$(S(t))_{t\geq 0}$ to $L^p(\T^2)$ is a strongly continuous semi-group on
the Banach space $L^p(\T^2)$ for $1 \leq p < + \infty$, and a weakly-$\star$ 
continuous semi-group on $L^\infty(\T^2).$ 

\medskip
 
Now that the functional set-up is in place, we can state our main two results. 
Their proofs combine comparison arguments, with on one hand averaging arguments 
(Theorem \ref{thm:plateau}) and on the other hand spectral arguments (Theorem
\ref{thm:Hor}). In both cases, the main step is to obtain a first (retarded time) 
pointwise lower bound on solutions starting as probability measures, and then 
rely on the following lemma which is nothing but an elementary special case of 
Doeblin's theorem~\cite{Dob40} (see also \cite{CanMis23} and the references 
therein for a recent account on so-called Doeblin-Harris type results). Note that 
the pointwise lower bounds \eqref{eq:lowerbP} and \eqref{eq:lowerbH} in 
Theorem~\ref{thm:plateau} and Theorem~\ref{thm:Hor} can 
be viewed as {\bf one sided} regularizing effects from $\mathcal{M}(\T^2)$ into $L^\infty(\T^2)$. 

\begin{lem}\label{lem:doeblin}
Let $(S(t))_{t\geq 0}$ be a conservative order preserving semi-group on 
$\mathcal{M}(\T^2)$ associated to the transition probabilities $(P_t)_{t\geq 0}$. Suppose that 
\begin{equation}\label{eq:uniformmixing}
P_{t_*}(x,y, \cdot) \geq \alpha_* \qquad \forall\, (x,y) \in \T^2,
\end{equation}
for some $t_* > 0$ and $0 < \alpha_* < 1.$ Then
\begin{equation}\label{eq:expopointwise}
P_t(x,y, \cdot) \geq 1 - C
\exp(-\rho t)
\qquad \forall\, t \geq 0, \forall\, (x,y) \in \T^2,
\end{equation}
where $C := 1/(1- \alpha_*)$ and $\rho := \log(C)/ t_* > 0.$
In particular, 
\begin{equation}\label{eq:exposemigroup}
\|S(t)\mu - \int_{\T^2} \, d\mu\|_{\mathcal{M}(\T^2)} \leq C 
e^{-\rho t}\|\mu - \int_{\T^2} \, d\mu\|_{\mathcal{M}(\T^2)}
\end{equation}
for any $\mu \in \mathcal{M}(\T^2)$ and any $t\geq 0.$
\end{lem}
\noindent Inequalities \eqref{eq:uniformmixing} and \eqref{eq:expopointwise} 
should be understood in the sense of measures, with the right-hand sides being 
uniformly distributed. They correspond therefore  to uniformly (w.r.t. to the 
source point $(x,y) \in \T^2$ and target point $(x',y') \in \T^2$) positive 
``pointwise'' lower bounds on the transition probabilities.

\smallskip\noindent
Our first sufficient condition to ensure the lower bound \eqref{eq:uniformmixing} 
and thus the exponential relaxation~\eqref{eq:exposemigroup} is: 

\begin{AP} 
There exist disjoint intervals  $I, J$ in $\T$ such $V$ is
constant both on $I$ and on $J$ with different values denoted
$V_I$ and $V_J$. The behavior of $V$ outside $I \cup J$, 
besides being bounded measurable, may be arbitrary.
\end{AP}
\noindent

\noindent We shall prove:

\begin{thm}\label{thm:plateau}
Assume that $(P)$ holds and set $\ell := \min(|I|, |J|)$.\\
Then at time 
\begin{equation}\label{def:tP}
t_P :=  \tfrac{1}{|V_I - V_J|} + \tfrac{3 + 2\ell^2}{8} 
\end{equation}
it holds
\begin{equation}\label{eq:lowerbP}
P_{t_P}(x,y,\cdot) \geq \alpha_P, \qquad \forall\, (x,y) \in \T^2,
\end{equation}
where   
$$
\alpha_P := (8\pi e)^{-\frac32}\exp(-\tfrac{\pi^2}{4}) \ell^2 
\exp(-\tfrac{\pi^2}{\ell^2|V_I - V_J|}).
$$
\end{thm}

We note that in the definition \eqref{def:tP} of $t_P$ the constant 
$(3 + 2\ell^2)/8$ is ``technical'' only and could be replaced by any 
(small) positive quantity, at the price of a lower value for $\alpha_P$. 
Instead, the constant $\frac{1}{|V_I - V_J|}$ has a clear interpretation 
as the minimal time needed for the process to spread across the whole torus 
in $y$ through  differential speed, and it is easy to show that no 
globally positive lower 
bound for $P_t$ can be derived for times $t$ shorter than this without 
additional assumptions on $V.$ More generally, defining the oscillation of $V$ by 
\begin{equation}\label{def:Osc}
{\it Osc}(V) := {\it ess\,sup}(V) - {\it ess\,inf}(V),
\end{equation}
it is clear that complete mixing requires at least a time $ t = \frac{1}{{\it Osc}(V)}.$

\smallskip

The proof of Theorem \ref{thm:plateau} is surprisingly short and 
elementary, with no spectral argument. It completely relies 
on comparison principles combined with averaging arguments
and the invariance by translation in $y$ of Equation~\eqref{eq:u}, and
to a much less extent on the explicit nature of solutions when $V$ is constant.

\medskip

\noindent Our other sufficient condition may be understood as a weak and 
local H\"ormander type condition:

\begin{AH} 
There exists an interval $I \subseteq \T$ and a constant $K \geq 1$ 
such that for any $0 < \eps < |I|$ and any interval $J \subseteq I$
of length $|J| \geq \eps$, 
$$
\eps \inf_{p,q \in \R} \int_J \big| PV(x) - px-q\big|^2\, dx \geq  
e^{-\frac{K}{\eps^2}}.
$$ 
\end{AH}
\noindent Here, $PV$ refers to any primitive of $V$ on $I$, and
the behavior of $V$ outside of $I$, besides being bounded measurable, 
may also be arbitrary. When $(H)$ holds, in particular $V$ may not  
have any plateau inside $I$; in rough terms $(H)$ requires 
that, {\it at least in a neighborhood of some point}, $V$ is 
locally nowhere too well approximated by 
a constant. A typical subclass arises when the restriction of $V$ 
to $I$ is $W^{1,1}(I)$ and satisfies a one sided inequality like 
$V' \geq \delta$ on $I$, for some $\delta > 0,$ or higher order variants. 

\medskip\noindent
When $V$ is not identically constant, and a fortiori when $(H)$ holds, the
quantity
\begin{equation}\label{def:w2-prop1}
    \omega_2(V) := \max \left\{ \int_{\T} V\varphi \ | \ \varphi \in Lip(\T,\R), |\varphi| \leq 1, 
|\varphi'|\leq 2\pi, \int_{\T}\varphi = 0\right\}
\end{equation}
is positive. We shall prove:

\begin{thm}\label{thm:Hor}
Assume that $(H)$ holds and set $\beta := 1 + 2\pi {\it Osc}(V)|I|^2.$\\
Then at time
\begin{equation}\label{def:tH}
t_H := \big( \tfrac{10\beta}{|I|\omega_2(V)}\big)^2 \big( 1 + \log(\beta) +
\tfrac{K}{|I|^2}\big) 
\end{equation}
it holds 
\begin{equation}\label{eq:lowerbH}
P_{t_H}(x,y,\cdot) \geq \alpha_H, \qquad \forall\, (x,y) \in \T^2,
\end{equation}
where 
$$
\alpha_H := \tfrac{|I|}{3} \exp(-\tfrac{\pi^2}{|I|^2} t_H).
$$
\end{thm}

Regarding the expression for $t_H$, it reflects the slightly more indirect 
method of proof and might be further from optimal than was $t_P.$ 

The proof of Theorem \ref{thm:Hor} relies on a comparison argument with 
the equivalent of \eqref{eq:u} on $I \times \T$ with Dirichlet boundary
conditions on $\partial I \times \T$, combined with spectral estimates 
(adapted from \cite{Wei21} to the Dirichlet case) providing a regularizing 
effect for the latter. The initial step hence consists in showing that sufficiently many particles
are driven to the region (here $I \times \T$) where a regularizing effect may be
expected. A somewhat related strategy was used in \cite{DoKoXuZh24} 
(see also \cite{FonMis24}) for a voltage-conductance model in neuroscience, 
which features complicated boundary 
conditions and a velocity field $V(x,y)$ depending both on $x,y$ and satisfying 
strong hypoellipticity conditions.

\medskip

As a consequence of Lemma \ref{lem:doeblin} and Theorem \ref{thm:plateau} and 
\ref{thm:Hor} therefore, if either $(P)$ or
$(H)$ holds then Equation \eqref{eq:u} induces exponential relaxation 
in $\mathcal{M}(\T^2)$ with explicit rates. Together with the weak contraction properties 
and interpolation, this also yields exponential relaxation in 
$L^p(\T^2)$, for initial data in $L^p(\T^2)$, for arbitrary $1 \leq p < \infty$ 
and with computable lower bounds on the rates.

\smallskip

Although the alternative $(P)$ or $(H)$ encompasses a large class of functions $V$, 
in particular all those that are piecewise $\mathcal{C}^1$ on $\T$, and much
more, it doesn't cover all bounded measurable functions.  
We leave open the question whether exponential relaxation always occurs  
even for measure initial data, as soon as $V \in L^\infty(\T)$ is not 
identically constant. 
A typical example of $V$ which does not satisfy any of $(H)$ or $(P)$, and which
might be a good candidate for further investigations, is given by the following. Write
any $x \in [0,1)$ in binary form as
$$
x = \sum_{k \geq 1} b_k 2^{-k}, \qquad b_k \in \{0, 1\},
$$
(uniquely avoiding infinite trailing ends of ones) and then set
$$
V(x) := \sum_{k \geq 1} a_k(-1)^{b_k},
$$
where $(a_k)_{k \geq 1}$ is a positive sequence converging ``extremely'' 
fast to zero. By construction it has no plateau so in particular it doesn't
satisfy $(P)$; yet it gets flatter and flatter ``locally everywhere'' at 
smaller scales, as given by the decay of $(a_k)_{k\geq 1}$,
so that it doesn't satisfy assumption $(H)$ if the latter 
decay is sufficiently strong. 

\medskip

As mentioned already, in the context of $L^2(\T^2)$ initial data
the non constancy of $V$ is a sufficient condition for 
exponential relaxation. We indeed have: 

\begin{prop}\label{prop:L2}
Let $V \in L^\infty(\T,\R)$ be non constant, and let $\omega_2(V) > 0$ as 
defined in \eqref{def:w2-prop1}. Then for any $u_0 \in L^2(\T^2)$ and any $t \geq 0$, 
\begin{equation}\label{eq:sgbound}
\big\|S(t)u_0 - \int_{\T^2}u_0\big\|_{L^2(\T^2)} \leq e^{\frac{\pi}{2} - \rho(V) t}
\big\|u_0 - \int_{\T^2}u_0\big\|_{L^2(\T^2)},
\end{equation}
where the mixing rate lower bound $\rho(V)$ is given by
\begin{equation}\label{eq:mixingcst}
\rho(V) := \left(\tfrac{\omega_2(V)}{2\pi(1 + Osc(V))}\right)^2.
\end{equation}
\end{prop}

Our proof of Theorem \ref{thm:Hor} is actually partly based on a 
variant of Proposition~\ref{prop:L2} adapted to the Dirichlet 
setting (see Proposition \ref{prop:v1} in Section \ref{sect:spectral}). 
The explicit lower bound \eqref{eq:mixingcst} for the mixing rate $\rho(V)$  
is not claimed to be sharp. As a matter of fact, a result very similar 
to Proposition \ref{prop:L2} was first obtained\footnote{Or at least follows 
immediately from \cite{Wei21}, as the author was mostly interested in 
the enhanced dissipation regime.} by 
D. Wei \cite{Wei21}. The lower bound expression obtained in \cite{Wei21} for the rate
would write   
$$
\tilde \rho(V) := \left(\Phi^{-1}(\omega_1(V))\right)^2, 
$$
where $\Phi(s) := 144s\tan(2s)$ for $s \in [0,+\infty)$ and     
$$
\omega_1(V) := \inf_{p,q \in \R} \int_0^1 \big|PV(x) - px - q\big|^2\, dx.
$$
That mixing rate is not claimed to be sharp either, 
both $\rho(V)$ and $\tilde \rho(V)$ behave somewhat 
similarly and measure a discrepancy for $V$ from 
being identically constant, their different values only reflect
slightly different strategies of proofs, which both rely 
on spectral arguments 
(uniform resolvent estimates for the operators 
$-\partial_{xx} + 2i\pi kV$ in 1D) after 
decomposing $u$ in the Fourier sectors in $y$ as 
$$
u(t, x, y) = \sum_{k \in \Z} \hat u_k(t,x) e^{2i\pi ky}.
$$
The fixed upper bound $e^{\frac{\pi}{2}}$ on the transient growth factor in 
\eqref{eq:sgbound} is derived from an explicit Gearhart-Pr\"uss type 
theorem for m-dissipative operators in Hilbert spaces also proved in 
\cite{Wei21}, and improving on earlier results of Helffer and Sj\"ostrand 
\cite{HelSjo10} (see also \cite{HelSjo21}).

\medskip

We finish this results section by mentioning the following easy shortcut 
that provides long-time mixing for $L^1(\T^2)$ initial data and arbitrary 
non constant $V$. Yet, it doesn't provide any rate, and also does not 
apply to Radon measures. For $u_0 \in L^1(\T^2)$ and $n \geq 0$, 
decompose  
$$
u_0 = u_{0,n} + v_n \ \text{ where } \
u_{0,n} := \max\left(-n, \min\left(u_0, n\right)\right), 
$$
so that by linearity
$$
S(t)u_0 - \int_{\T^2}u_0 = S(t)u_{0,n} - \int_{\T^2}u_{0,n} + S(t)v_n -
\int_{\T^2} v_n.
$$
Applying Proposition \ref{prop:L2} for $u_{0,n} \in L^2(\T^2)$ and the semi-group 
contraction property for $v_n \in L^1(\T^2)$, we obtain
$$
\Big\|S(t)u_0 - \int_{\T^2}u_0\Big\|_{L^1(\T^2)} \leq e^{\frac{\pi}{2} -
\rho(V)t}\Big\|u_{0,n} -
\int_{\T^2} u_{0,n}\Big\|_{L^2(\T^2)} + \Big\|v_n -
\int_{\T^2}v_n\Big\|_{L^1(\T^2)},
$$
from which it easily follows by taking first $n$ sufficiently large, and then
$t$ sufficiently large given $n$, that 
$$
\big\|S(t)u_0 - \int_{\T^2}u_0\big\|_{L^1(\T^2)} \to 0 \quad \text{ as } \ t \to
+\infty.
$$

\medskip

To close this introduction, let us mention that important contributions to 
the question of relaxation to equilibrium for problems involving the 
interplay between a symmetric (diffusive) part, and a skew-symmetric 
(conservative) one, have been obtained by many authors in the past, both in linear 
and nonlinear contexts much more involved than~\eqref{eq:u}.  
Kinetic theory has certainly played a major role, 
ranging from the linear Kolmogorov, BGK or Fokker-Planck equations, to the
nonlinear Vlasov-Poisson and Boltzmann equations. It is not our aim to try
to account for these faithfully, and we refer instead to literature 
reviews as can be found in, e.g., \cite{Herau18} and \cite{BeFaLeSt22}. The specificity  
of kinetic models is that the corresponding (free) transport field is the smooth 
and non degenerate $V(x) = x$ (we should 
actually write $v$ instead of $x$ here since it corresponds to a velocity variable). 
Recent contributions (see e.g. \cite{GoImMoVa19, HeHeHuMo22, AnHeGuLoMoRe24}) 
have considered non-smooth or degenerate coefficients, 
but therefore on the diffusive part rather than on the transport one.  
Fluid mechanics, as already mentioned, has also been an important contributor, 
including the study of long time relaxation of vortices, enhanced dissipation, 
and the understanding of the transition threshold to instabilities for perturbed  
stationary flows in the context of the Navier-Stokes equation. About these, we refer to 
the nice expository paper \cite{Gallay21} (in french) after \cite{LiWeiZha20}, the references therein,  
and to the impressive memoir \cite{CheWeiZha24}. 
Our modest personal initial motivation to these questions instead came from models 
in neuroscience, where the 
diffusion naturally arises only for some of the variables, and we settled
first on the model problem \eqref{eq:u} for its simplicity. 
We hope that any progress related to it might provide insights into the fine 
structure of mixing mechanisms for other degenerate parabolic-transport systems. 
We conjecture in particular that for operators of the form $X_0 + \sum_{i=1}^r X_i^*X_i$ 
on manifolds the minimal time for mixing (in the sense of obtaining uniform pointwise 
positive lower bounds on $P_t$) should be equal to the minimal control time of system 
\eqref{eq:controlsys} with the additional constraint that $\xi_0 \equiv 1$ (in other 
words, for which diffusion directions can be followed at arbitrary large speed, 
while transport directions are followed at unit speed), the required level of 
regularity for the $X_i$ is still unclear though; see also Subsection \ref{sect:hj}.

\medskip

The plan of the paper is as follows: Section \ref{sect:plateau} is devoted to the 
proof of Theorem \ref{thm:plateau}, which is self-contained. 
In Section \ref{sect:spectral} we derive spectral estimates for the 1D linear 
operators associated to Equation~\eqref{eq:u} 
after decomposition into Fourier modes in $y$. These allow us in particular to 
present the proof of Proposition \ref{prop:L2}. Section \ref{sect:Hor} is devoted to
the proof of Theorem \ref{thm:Hor}, which combines elements in Section
\ref{sect:spectral} with comparison and one sided smoothing estimates.  Finally, in
Section \ref{sec:fourretout} we gather a number of aside remarks:\\
\indent {\it i)} On the 
(non) smoothing properties of \eqref{eq:u} in case of plateaux and/or jumps, in
particular in regard to the explicit example by Lévy \cite{Levy39} of the so-called 
arcsine law for the time spent by a Brownian motion in a half-line.\\
\indent {\it ii)} On the
applicability of other existing strategies to our problem, in particular 
the use of Lyapunov type functionals constructed as perturbations of the natural
norms with the help of well chosen commutators, similar to the ones involved in 
the subelliptic estimates in the hypoelliptic case. These methods have been associated with 
the name of {\it hypocoercity}, and have been successful in an number of 
contexts.\\
\indent{\it iii)} 
The account of a completely different approach to mixing, through an optimal 
control problem partly related to \eqref{eq:controlsys}, that seemed promising 
an yields an alternative construction of fundamental solutions for the 
Kolmogorov equation, but which we didn't manage to push all the way through in 
the case of Equation~\eqref{eq:u} for general $V$'s.

\numberwithin{equation}{section}

\section{At least two plateaux}\label{sect:plateau}

The purpose of this section is to present the proof of Theorem
\ref{thm:plateau}, it is self-contained.

On any plateau of $V$, that is any positive length interval $I \subset \T$ on 
which $V(x) \equiv V_I \in \R$ 
is a constant, Equation~\eqref{eq:u} reduces to a heat equation in the $x$ variable 
with a constant drift in the $y$ variable, i.e. 
\begin{equation}
    \partial_tu-\partial_{xx}u+V_I\partial_yu=0,\qquad x\in I,y\in\mathbb{T}.
\end{equation}
Using the comparison principle between the  
original equation \eqref{eq:u} on $\T^2$ and its restriction to $I \times \T$ 
with homogeneous Dirichlet boundary conditions on $\partial I \times \T$ (or
equivalently the process \eqref{eq:sp} and the corresponding one where 
trajectories exiting the strip $I \times \T$ are killed), it
follows that the transition probabilities defined in \eqref{eq:tp}
satisfy\footnote{
Although $P_t(x, y, \cdot)$ need not be absolutely continuous with respect
to the Lebesgue measure on $\T^2$ even for $t > 0$ (this is indeed the case 
when $V$ has a plateau), for the ease of notation, especially of 
multi-steps Chapman-Kolmogorov equations, we shall sometimes {\it denote} 
it by $P_t(x,y, x',y')$ or $P_t(x,y,x',y') dx'dy'$ 
in the sequel when this does not lead to a confusion.}
\begin{equation}\label{eq:compare}
P_t(x,y,x',y') dx'dy' \geq \delta_{y'= y + V_It}
\mathbbm{1}_{I}(x)\mathbbm{1}_{I}(x') G_t^I(x, x')dx',
\end{equation}
where $G_t^I$ refers to the fundamental
solution for the 1D heat equation on $I$ with homogeneous 
Dirichlet boundary conditions on $\partial I.$ 

From standard regularity theory (Harnack inequality) and a scaling argument to reduce 
to the case of a unit interval, we also have regarding $G_t^I$ the classical estimate
\begin{equation}\label{eq:greenbound1d}
G_t^I(x, x') \geq \frac{c}{|I|}e^{-\frac{\pi^2t}{|I|^2}} \qquad \forall\, t \geq
\frac{|I|^2}{8}, \forall\, x, x' \in I',
\end{equation}
where $I' = \{x \in I \text{ s.t. } d(x,I^c) \geq |I|/4\}$ and 
$
c \geq 1 - 2 \sum_{k \geq 2} e^{-\frac{\pi^2}{8}(k^2-1)} \geq \frac12,
$
the lower bound for $c$ being easily established through a decomposition 
into spectral basis. The main point here is to establish a uniformly positive bound on $I'$.

\medskip

In the sequel of this section we assume that $V$ possesses at least two plateaux
$I$ and $J$ at different heights $V_I \neq V_J$. Without loss of generality
we assume for convenience below that $|I| = |J| =: \ell,$ but
this could easily be removed at the price of slightly longer expressions 
in some bounds.

\medskip

For arbitrary $T = t_0 + t_1 + t_2$, iterated use of the Chapman-Kolmogorov equation \eqref{eq:CK} yields
\begin{equation}\label{eq:itersk}
P_T(x_0,y_0, \cdot)
= \iint P_{t_0}(x_0, y_0, dx_1dy_1)P_{t_1}(x_1, y_1, dx_2dy_2)P_{t_2}(x_2, y_2,
\cdot).
\end{equation}
Starting from $(x_0,y_0)$ with $x_0 \in I'$, we shall only consider 
process trajectories that stay in the strip $I \times \T$ up to time $t_0$, where they arrive at 
some $(x_1,y_1)$ with $x_1 \in I'$, then traverse to some $(x_2, y_2)$ with 
$y_2 \in J'$ in a $t_1$ time step, and finally stay in the strip $J \times \T$ for an
additional $t_2$ time step to reach some $(x,y)$ at time $T$ with $x \in J'$. Here 
$J' = \{x \in J \text{ s.t. } d(x,J^c) \geq |J|/4\}$ is defined similarly as~$I'$. In other 
words, provided $t_0, t_2 \geq \frac{\ell^2}{8}$ we make use of \eqref{eq:compare} and 
\eqref{eq:greenbound1d} to bound 
$P_{t_0}$ and $P_{t_2}$ in \eqref{eq:itersk}. This yields 
\begin{equation}\label{eq:salade}
P_T(x_0,y_0,x,y)\\ 
\geq \frac{1}{4\ell^2}e^{-(\frac{t_0 + t_2}{\ell^2})\pi^2}
 d\mu(y)  \mathbbm{1}_{I'}(x_0) \mathbbm{1}_{J'}(x) 
\end{equation}
where\footnote{Note that although $P_t(x,y, \cdot)$ is barely a Radon measure, 
expressions like $P_t(x,y, x', y' + a)$ are perfectly meaningful 
and simply refer to the push forward of $P_t(x, y, \cdot)$ by a translation 
by $a$ in the $y'$ direction.}
$$
d\mu(y) = \int_{I' \times J'}
P_{t_1}(x_1, y_0 + V_I t_0, x_2, y - V_J t_2)\, dx_1 dx_2
$$
actually depends on $y_0$, $t_0$, $t_1$ and $t_2$.

The apparent difficulty to
obtain a pointwise lower a bound on $d\mu(y)$ (which might only be a measure at
this stage) stems from the loss of control of the process position in the 
$y$ direction during the traverse between the two strips mentioned above, since
we do not control $V$ there. In order to obtain such a bound, we shall use an averaging
procedure together with invariance by translation in $y$ of Equation~\eqref{eq:u}.

More precisely, for a given arbitrary $s \in S :=[0, \frac{1}{|V_I - V_J|}]$ we set 
$$
t_0(s) = \frac{|I|^2}{8} + s, \quad t_1 = \frac18, \quad  t_2(s) =
\frac{|J|^2}{8} + \frac{1}{|V_I -
V_J|} - s, 
$$ 
so that $T = t_0(s) + t_1 + t_2(s)$ does not depend on $s$. By construction 
$t_0(s),t_2(s)\geq \frac{\ell^2}{8}$ and thus \eqref{eq:salade} holds for every 
$s\in S$ with $d\mu(y)$ depending on $s$. Averaging of
\eqref{eq:salade} for $s$ in $S$  then leads to 
\begin{equation}\label{eq:salade2}
P_T(x_0,y_0,x,y)\\ 
	\geq \frac{1}{4\ell^2}e^{-\frac{\pi^2}{\ell^2}(T-\frac{1}{8})}
 d\nu(y)  \mathbbm{1}_{I'}(x_0) \mathbbm{1}_{J'}(x)\, dx 
\end{equation}
where
$$
d\nu(y) = \int_{I' \times J'}\int_S \frac{1}{|S|}
P_{\frac18}\Big(x_1, y_0 + V_It_0(s), x_2, y - V_Jt_2(s)\Big)\, ds dx_1 dx_2.
$$
By invariance by translation of Equation~\eqref{eq:u} in $y$ and an affine change of
variable, we compute that
\begin{equation}\label{eq:dnu}\begin{split}
  &\frac{1}{|S|} \int_S P_{\frac18}\big(x_1, y_0+ V_It_0(s), x_2, y -V_Jt_2(s)\big)\, ds\\
=\ &\frac{1}{|S|} \int_S P_{\frac18}\Big(x_1, y_0, x_2, y -(V_I + V_J)\ell^2 - \tfrac{V_J}{|V_I -
V_J|} + (V_J- V_I)s\Big)\, ds\\
=\ & \int_\T P_{\frac18}(x_1, y_0, x_2, y')\, dy',   
\end{split}\end{equation}
and in particular that $d\nu$ is a uniform distribution in $y$ on $\T$ (it does not depend
on $y_0$ either). In the change of variable we use crucially $V_J\neq V_I$ and 
$|S|=\frac{1}{|V_I-V_J|}$. Besides, since integration of Equation~\eqref{eq:u} in the $y$ 
variable yields the heat equation on $\T$, 
$$
\int_\T P_{t}(x_1, y_0, x_2, y')\, dy' = G_t^\T(x_1, x_2)
$$
where $G_t^\T(\cdot,\cdot)$ refers to the fundamental solution of 
the 1D heat equation on $\T$. Observe that  
\begin{equation}\label{eq:borneGtore}
\kappa_0 := \inf_{x,x' \in \T} 
	G_\frac18^\T(x,x') \geq G_\frac18(\tfrac 12) = \sqrt{\tfrac{2}{e\pi}},
\end{equation}
where
$$
G_t(x) := \frac{1}{\sqrt{4\pi t}} e^{-\frac{x^2}{4t}}
$$
refers to the fundamental solution of the 1D heat equation
on $\R,$ and the choice $t = \frac{1}{8}$ was made because it optimizes the
resulting lower bound $\kappa_0$.  
Gathering all information obtained so far, we have thus obtained 
\begin{equation}\label{eq:saladecuite}
P_T(x_0, y_0, x,y) \geq \frac{\kappa_0}{16}
	e^{-\frac{\pi^2}{\ell^2}(T-\frac{1}{8})}
\mathbbm{1}_{J'}(x), \quad \forall\, x_0 \in I', \ \forall\, y_0 \in \T.
\end{equation}

In order to apply the Doeblin argument of Lemma \ref{lem:doeblin}, we still 
wish to remove the restriction
on $x_0$ in the previous inequality. For that purpose, we simply use again
the Chapman-Kolmogorov equation \eqref{eq:CK} and estimate
\begin{equation*}\begin{split}
P_{T+\frac18}(x_0, y_0, x,y) &=  \int P_\frac18(x_0,y_0, dx_1dy_1)P_T(x_1,y_1,x,y)\\
&\geq \frac{\kappa_0}{16} e^{-\frac{\pi^2}{\ell^2}(T-\frac{1}{8})}
P_\frac18(x_0,y_0, I' \times \T) \mathbbm{1}_{J'}(x) 
\end{split}\end{equation*}
with 
$$
P_\frac18(x_0,y_0, I' \times \T) = \int_{I'} G_\frac18^\T(x_0, x')\, dx' \geq
\kappa_0|I'|,
$$
which now yields a lower bound independent of $(x_0, y_0)$. From the stochastic process 
point of view, starting from any position $x_0\in\mathbb{T}$ there is a positive 
probability that at time $t_3=\frac{1}{8}$ the Wiener process is inside $I'$.

We can finally globalize our lower bound (i.e. remove the restriction $x \in J'$) by
a last application of the Chapman-Kolmogorov formula,
\begin{equation*}\begin{split}
P_{T+\frac14}(x_0,y_0,x,y) &= \int P_{T+\frac18}(x_0,y_0, dx_1dy_1)P_\frac18(x_1,y_1,x,y)\\
	&\geq \frac{\kappa_0^2}{32} \ell e^{-\frac{\pi^2}{\ell^2}(T-\frac18)}\int_{J'\times
\T}P_\frac18(x_1,y_1,x,y) dx_1dy_1\\
	& =\frac{\kappa_0^2}{32} \ell e^{-\frac{\pi^2}{\ell^2}(T-\frac18)}
	\int_{J'}G_\frac18^\T(x_1,x)\, dx_1 \geq \frac{\kappa_0^3}{64}\ell^2
	e^{-\frac{\pi^2}{\ell^2}(T-\frac18)}
\end{split}\end{equation*}
where here we have used the fact that integration in the $y$ direction for the
initial data (i.e. here in $y_1$) also reduces to the one dimensional heat 
equation on $\T.$ With the definition $t_P := T + \frac14$, this yields  
$$
P_{t_P}(x_0,y_0,x,y) \geq \frac{\kappa_0^3}{64}\ell^2 e^{-\frac{\pi^2}{4}}
e^{-\frac{\pi^2}{\ell^2} \frac{1}{|V_I - V_J|}},
$$
which is precisely \eqref{eq:lowerbP}. This concludes the proof of 
Theorem \ref{thm:plateau}.\qed

\section{Resolvent and semi-group estimates in 1D}\label{sect:spectral}

Let $I = [a,b] \subset \R$ be a bounded interval and $V \in L^\infty(I, \R).$\\
We denote by $A$ the linear operator
$$
A u := -\partial_{xx} u + iV(x) u
$$ 
on $D(A) \subset L^2(I)$ where either  
$$
D(A) := \Big\{ u \in H^2(I, \C) \text{ s.t. } u(a) = u(b), u'(a) = u'(b) \Big\}
$$
(periodic boundary conditions), or 
$$
D(A) := \Big\{ u \in H^2(I, \C) \text{ s.t. } u(a) = u(b) = 0 \Big\}
$$
(homogeneous Dirichlet conditions).\\
In both cases, we denote by $(\lambda_i)_{i \geq 1}$ the increasing 
sequence of real eigenvalues for the {\it Laplace} operator on $I$ 
with the same boundary conditions. We also denote by $e_1$ the  
real-valued $L^2$-normalized eigen-function associated to $\lambda_1$, i.e.
\begin{equation}\label{def:eigenf}
    e_1(x) =|I|^{-1/2} \left\{ 
\begin{array}{rl}
1 & (\text{periodic case}),\\
\sqrt{2}\sin(\pi \frac{x-a}{b-a}) & (\text{Dirichlet case}).
\end{array}
\right.
\end{equation}
For any $u \in D(A)$ and any $z \in \C$ 
with ${\rm Re}(z) < \lambda_1$, we have the immediate estimate
$$
{\rm Re}\left( \langle Au - zu, u \rangle \right) = \int_I |\partial_x u|^2 
- {\rm Re}(z) |u|^2 \,dx \geq \big(\lambda_1 -
{\rm Re}(z)\big) \|u\|^2. 
$$ As a consequence the left open half-plane $\{ {\rm Re}(z) < \lambda_1\}$ is contained 
in the resolvent set of $A$ and 
$$\|(A - z)^{-1}\| \leq \frac{1}{\lambda_1 - {\rm Re}(z)} \quad \text{  
for all } {\rm Re}(z) < \lambda_1,$$
that is $A - \lambda_1$ is m-accretive.

\medskip
Let  
\begin{equation}\label{eq:rlambda1}
r(\lambda_1) := \inf \Big\{ \|(A - z)u\|,\ u \in D(A),\ \|u\|=1,\ {\rm Re}(z) =
\lambda_1\Big\}
\end{equation}
[in result statements we shall write it $r(\lambda_1, V)$ to stress its dependence on
$V$].\\[2pt]
If $r(\lambda_1) > 0$,  
it follows from standard properties of resolvents that the left open half-space 
$\{{\rm Re}(z) < \lambda_1 + r(\lambda_1)\}$ is fully included in the resolvent 
set of $A$ and that 
\begin{equation}
\frac{1}{r(\lambda_1)} = \sup_{{\rm Re}(z) = \lambda_1} \|(A - z)^{-1}\|. 
\end{equation}Besides, since $A - \lambda_1$ is m-accretive, D. Wei's explicit semi-group 
estimate \cite{Wei21} applies and we have the semi-group bound
\begin{equation}\label{eq:weiineq}
\|e^{-tA}\| \leq e^{\frac{\pi}{2} - (\lambda_1 + r(\lambda_1))t}, \qquad \forall\,
t \geq 0.
\end{equation}

Our aim in this section is to show that $r(\lambda_1) > 0$   
whenever $V$ is not identically constant (this condition is 
also clearly necessary). More importantly, we are interested 
in explicit lower bounds in terms of $V$. We stress that our results apply both to the 
periodic case and the Dirichlet case, and we present them simultaneously.

\medskip

For our first result, we consider the space
%$$
\begin{equation}\label{def:L}
    L := \left\{ \varphi \in Lip(I,\R),\ \int_I\varphi e_1^2 = 0, \ \varphi e_1(a) =
\varphi
e_1(b) \right\}
\end{equation}
%$$
endowed with the norm
$$
\|\varphi\|_L := \max\left(\|\varphi\|_\infty, \frac{|I|}{2\pi}\, \|\varphi'\|_\infty\right).
$$
Recall that $e_1$ is the first eigenfunction defined in \eqref{def:eigenf}. 
In the periodic case, $L$ therefore consists of periodic Lipschitz functions whose 
integrals on $I$ vanish. In the Dirichlet case, $L$ just consists of Lipschitz 
functions that are orthogonal to $e^2_1$ with respect to the usual $L^2$ inner 
product, since in this case $e_1(a)=e_1(b)=0$ and thus the second constraint in 
\eqref{def:L} is automatically satisfied.
 
We define the quantity\footnote{The unit ball of $L$ being compact in, e.g.,
$L^\infty(I)$, the maximum is clearly achieved.}
\begin{equation}\label{eq:defomega2}
	\omega_2(V) := \max \left\{\int_I V\varphi e_1^2,\ \varphi \in L,
	\|\varphi\|_L \leq 1\right\}.
\end{equation}

\begin{prop}\label{prop:v1} 
Let $V \in L^\infty(I,\R)$ be non constant. Then $\omega_2(V) > 0$ and we have  
\begin{equation}\label{eq:prop1}
r(\lambda_1, V) \geq \frac{\omega_2(V)^2}{18} 
\left( \frac{\pi^2}{|I|^2} + \frac{|I|^2}{\pi^2} {\it Osc}(V)^2\right)^{-1} > 0,
\end{equation}
where ${\it Osc}(V) := {\it ess\,sup}(V) - {\it ess\,inf}(V)$.
\end{prop}

\begin{proof}
The fact that $\omega_2(V) > 0$ when $V$ is non constant is immediate.
Indeed, by the du Bois-Reymond lemma there exists 
$\phi \in \mathcal{D}(I)$ such that $\int_I \phi = 0$ and $\int_I
V \phi \neq 0.$ Since $e_1$ is smooth and positive in the interior of
$I$, we may set $\varphi := \phi / e_1^2$, and we have $\varphi \in
\mathcal{D}(I)$, $\int_I \varphi e_1^2 = 0$, and $\int_I
V \varphi e_1^2 \neq 0.$ The conclusion follows by multiplying
$\varphi$ by a suitable non zero constant. 

We now turn to the estimate of $r(\lambda_1).$ We write any 
$z \in \C$ s.t. ${\rm Re}(z) = \lambda_1$  as 
$z = \lambda_1 + is$, with $s \in \R.$ For $f \in L^2(I)$, 
we wish to estimate a solution $u \in D(A)$ of the equation 
\begin{equation}\label{eq:resolv}
-\partial_{xx}u - \lambda_1 u + i(V - s)u = f.
\end{equation}
For that purpose, we decompose $u$ as 
\begin{equation}\label{eq:decomp-u}
    u = \langle u, e_1\rangle e_1 + v,
\end{equation} and let $C_1 := \langle u, e_1 \rangle.$ Since $v$ is orthogonal to $e_1$ we have
\begin{equation}\label{eq:lambda2}
\int_I |\partial_x v|^2 \geq \lambda_2 \int_I |v|^2,
\end{equation}
where recall $\lambda_2>\lambda_1$ is the second eigenvalue of the corresponding 
\textit{Laplace} operator. Therefore
\begin{equation}\label{eq:spectralgap1}
\int_I |\partial_x v|^2 - \lambda_1 |v|^2 \geq \big(1 -
\frac{\lambda_1}{\lambda_2}\big)
\int_I |\partial_x v|^2 \geq (\lambda_2 - \lambda_1) \int_I
|v|^2.
\end{equation}
Multiplying equation \eqref{eq:resolv} by the conjugate $\bar u$ and then retaining the real
part, we obtain 
$$
\int_I |\partial_x u|^2 - \lambda_1 |u|^2 \leq \|f\|\: \|u\|. 
$$
Since 
$$
\int_I |\partial_x u|^2 - \lambda_1 |u|^2 = \int_I |\partial_x v|^2 - \lambda_1
|v|^2,
$$
from \eqref{eq:spectralgap1} we deduce the estimates
\begin{equation}\label{eq:controlv}
	\|v\|^2 \leq \frac{1}{\lambda_2 - \lambda_1}\|f\|\: \|u\|\quad \text{
	and } \quad
	\|\partial_x v\|^2 \leq \big(1 - \frac{\lambda_1}{\lambda_2}\big)^{-1}
\|f\|\: \|u\|.
\end{equation}
If instead we retain the imaginary part, we obtain
\begin{equation}\label{eq:controlweightedu}
\left| \int_I (V - s) |u|^2 \right| \leq \|f\|\: \|u\|.
\end{equation}

We shall distinguish two cases for $s$.

\smallskip

\noindent{\bf Case 1:} $\|V - s\|_\infty \geq Osc(V) + \omega_2(V).$\\
Then in \eqref{eq:controlweightedu} the 
weight $(V - s)$ is of constant sign a.e. on $I$ and moreover 
by definition of $Osc(V)$ 
$$
|V - s| \geq \omega_2(V) \qquad a.e. \text{ on } I.
$$
We therefore deduce from \eqref{eq:controlweightedu} that
\begin{equation}\label{eq:case1}
\|u \| \leq \frac{1}{\omega_2(V)} \|f\|.
\end{equation}

\medskip
\noindent{\bf Case 2:} $\|V - s\|_\infty \leq Osc(V) + \omega_2(V).$\\
In this case, the difficulty is that $V - s$ possibly changes sign. We aim to use the 
decomposition \eqref{eq:decomp-u} and to control $u$ by (controlling) 
$C_1=\langle u, e_1 \rangle$. To do so, we consider some $\phi \in Lip(I, \R)$ 
such that $\phi e_1(a) = \phi e_1(b)$,
multiply equation \eqref{eq:resolv} by $\phi \bar u$ and then retain the
imaginary part. Using that $e_1$ is real, this yields the identity 
\begin{equation}\label{eq:weightedident}\begin{split}
&{\rm Im}\Big(
C_1 \int_I \partial_x e_1 \bar v \phi' + \overline{C_1} \int_I e_1 \partial_x v \phi'
+ \int_I \bar v \partial_x v \phi' \Big) + |C_1|^2 \int_I (V - s)\phi e_1^2\\
& + \int_I (V - s) \phi \left[ |v|^2 + 2 {\rm Re}(C_1 e_1\bar v)\right]
= {\rm Im}\Big( \int_I f\phi \bar u\Big),
\end{split}\end{equation}
from which it follows that
\begin{equation}\label{eq:controlC1}\begin{split}
|C_1|^2 \Big| \int_I (V-s)\phi e_1^2 \Big| \leq 
& \ \|f\| \|u\| \|\phi\|_\infty\\
	& + |C_1| \|\phi'\|_\infty \Big[ \sqrt{\lambda_1}\|v\| + \|\partial_x v\| \Big] \\
& + \|\phi'\|_\infty \|\partial_x v\| \|v\|\\
&+ \|V - s\|_\infty \|\phi\|_\infty \Big[ 2 |C_1| \|v\| + \|v\|^2\Big],
\end{split}\end{equation}
where we used the Cauchy-Schwarz inequality with $\|e_1\|=1$ and 
$\|\partial_x e_1\|=\sqrt{\lambda_1}$. We now assume moreover that $\phi$ is a maximizer 
for $\omega_2(V)$ in 
\eqref{eq:defomega2}, that is $\int_I \phi e_1^2 = 0$,
$\|\phi\|_\infty \leq 1$, $\|\phi'\|_\infty  \leq 2\pi/|I|=\sqrt{\lambda_2}$ and 
$\int V \phi e_1^2 = \omega_2(V).$
In particular,
$$
\int_I (V-s)\phi e_1^2 = \int_I V \phi e_1^2 = \omega_2(V).
$$
Rewriting \eqref{eq:controlC1} taking into account the latter as well as 
\eqref{eq:controlv} and the assumption $\|V-s\|_\infty \leq Osc(V) +\omega_2(V)$ 
we then obtain 
\begin{equation}\label{eq:controlC1bis}\begin{split}
	|C_1|^2 \leq & \frac{1}{\omega_2(V)}\Bigg[ \Big( 1 + 
	\frac{\lambda_2}{\lambda_2-\lambda_1}\big(1 + \frac{(
	Osc(V) + \omega_2(V))}{\lambda_2}\big)\Big)
	\|f\| \|u\|\\
	& + \frac{\lambda_2}{\sqrt{\lambda_2-\lambda_1}}\Big(1 +
	\sqrt{\frac{\lambda_1}{\lambda_2}} + \frac{2(Osc(V) + \omega_2(V))}{\lambda_2}\Big)
	|C_1| \|f\|^\frac12 \|u\|^\frac12 \Bigg]. 
\end{split}
\end{equation} 
Since $\|u\|^2 = |C_1|^2 + \|v\|^2$, adding \eqref{eq:controlC1bis} with 
the first inequality in \eqref{eq:controlv}, and then replacing $|C_1|$ by $\|u\|$ on the
r.h.s. of the resulting inequality we finally obtain, after division by $\|u\|$, 
\begin{equation}\label{eq:quadratic0}
\|u\| \leq \alpha \|f\| + \beta \|f\|^\frac12 \|u\|^\frac12,
\end{equation}
where
$$
\alpha = \frac{1}{\omega_2(V)}\Big( 1 + 
    \frac{\lambda_2}{\lambda_2- \lambda_1}\Big) + \frac{Osc(V) + \omega_2(V)}{\omega_2(V)
    (\lambda_2 - \lambda_1)} +\frac{1}{\lambda_2 - \lambda_1} 
$$
and
$$
\beta = \frac{1}{\omega_2(V) \sqrt{\lambda_2-\lambda_1}} \Big(\lambda_2(1 +
	\sqrt{\tfrac{\lambda_1}{\lambda_2}})  + 2(Osc(V) + \omega_2(V))\Big).
$$
Analyzing \eqref{eq:quadratic0} as an inequality for a quadratic polynomial in
$\|u\|^\frac{1}{2}$, and computing its largest root, it immediately follows 
that  
\begin{equation}\label{eq:case2}
\|u\| \leq (4\alpha + \beta^2) \|f\|.
\end{equation}
In the remaining part of this proof we use the short hand notations $\omega$ for
$\omega_2(V)$ and $O$ for ${\it Osc}(V).$ In both boundary conditions 
cases\footnote{As a matter of fact, for the periodic case we have 
$\lambda_1=0$ and $\lambda_2=(2\pi/|I|)^2$, while for the Dirichlet case we 
have $\lambda_1=(\pi/|I|)^2$ and $\lambda_2=(2\pi/|I|)^2$.}, we have $\lambda_2 = (2\pi/|I|)^2$,
$\lambda_2 / (\lambda_2 - \lambda_1) \leq \frac{4}{3},$ and $\lambda_1 /
\lambda_2 \leq \frac{1}{4}$. In particular
$$
\alpha \leq \frac{7}{3 \omega} + \frac{|I|^2}{3\pi^2}\left(2 + 
\frac{O}{\omega}\right)
$$
and
$$
\beta \leq \frac{2\sqrt{3}\pi}{\omega|I|} + \frac{2|I|}{\sqrt{3}\pi}\left( 1 +
\frac{O}{\omega}\right),
$$
and therefore
\begin{equation}\label{eq:boundalphabeta}
4\alpha + \beta^2 \leq \frac{28}{3\omega} + \frac{4|I|^2}{3\pi^2}
\big(3 + 3\frac{O}{\omega} + (\frac{O}{\omega})^2\big)  
+ \frac{12\pi^2}{|I|^2\omega^2} + \frac{8}{\omega}(1 +  \frac{O}{\omega}).
\end{equation}
Note that from the definition \eqref{eq:defomega2} of $\omega$, and using
the fact that $e_1$ is $L^2$ normalized, it follows that in all
circumstances $\omega \leq \frac{1}{2}O.$ 
In \eqref{eq:boundalphabeta}, we may therefore estimate
$$
\frac{1}{\omega} \leq \frac{1}{2}\frac{\pi}{|I|\omega} \frac{|I|}{\pi}
\frac{O}{\omega} \leq \frac{1}{8} \frac{\pi^2}{|I|^2\omega^2} +
\frac{1}{2} \frac{|I|^2}{\pi^2}(\frac{O}{\omega})^2,
$$
$$
\big(3 + 3\frac{O}{\omega} + (\frac{O}{\omega})^2\big) \leq \frac{13}{4}(
\frac{O}{\omega})^2,
$$
and 
$$
\frac{8}{\omega}(1+ \frac{O}{\omega}) \leq 12 \frac{1}{\omega} \frac{O}{\omega} 
\leq 4\frac{\pi^2}{|I|^2\omega^2} + 9
\frac{|I|^2}{\pi^2} (\frac{O}{\omega})^2,
$$
which yields
\begin{equation}\label{eq:boundalphabeta2}
4\alpha + \beta^2 \leq (16 + \frac{7}{6}) \frac{\pi^2}{|I|^2\omega^2} +
18\frac{|I|^2}{\pi^2} (\frac{O}{\omega})^2 \leq \frac{18}{\omega^2} 
\left( \frac{\pi^2}{|I|^2} +
\frac{|I|^2}{\pi^2} O^2\right).
\end{equation}

We are now in position to end the proof in all cases. Indeed, note that in 
Case 1 estimate \eqref{eq:case1} we may transform as above  
$$
\frac{1}{\omega}  \leq \frac{1}{2}\frac{1}{\omega}\frac{O}{\omega} 
\leq 
\frac{1}{4}\frac{\pi^2}{|I|^2\omega^2} + 
\frac{1}{4}\frac{|I|^2}{\pi^2}(\frac{O}{\omega})^2
\leq
\frac{1}{4\omega^2} \left( \frac{\pi^2}{|I|^2} +
\frac{|I|^2}{\pi^2} O^2\right),
$$
so that combining \eqref{eq:case1} with \eqref{eq:case2} and
\eqref{eq:boundalphabeta2} the conclusion \eqref{eq:prop1} follows.  
\end{proof}

\begin{rem}
In the case of periodic boundary conditions we
have $\lambda_1 = 0$ and in the proof of Proposition \ref{prop:v1} we may slightly 
improve the bounds on $\alpha$ and $\beta$; in particular for $I = \T$  we
obtain
$$
\alpha \leq \frac{2}{\omega} + \frac{1}{4\pi^2}(2 + \frac{O}{\omega}), \qquad
\beta \leq \frac{2\pi}{\omega} + \frac{1}{\pi}(1 + \frac{O}{\omega}),
$$
from which it follows, taking once more into account the bound $\omega \leq
\frac{1}{2}O$, that
$$
4\alpha + \beta^2 \leq \frac{4\pi^2}{\omega^2} + 10 \frac{O}{\omega^2} +
\frac{13}{4\pi^2}(\frac{O}{\omega})^2 \leq \frac{4\pi^2(1 + O)^2}{\omega^2},
$$
and then that
\begin{equation}\label{eq:boundimproved}
r(\lambda_1, V) \geq \left( \frac{\omega_2(V)}{2\pi(1 + {\it Osc}(V))}\right)^2.
\end{equation}
\end{rem}

\medskip
We are now in position to present the
\begin{proof}[Proof of Proposition \ref{prop:L2}]
As mentioned in the introduction we decompose the solution in Fourier series
$$
u(t, x, y) = \sum_{k \in \Z} \hat u_k(t,x) e^{2i\pi ky},
$$
so that $u_k$ solves the equation
$$
\partial_t u_k - \partial_{xx}u_k + 2i\pi kV(x)u_k = 0.
$$
For $k = 0$ we simply get that the mean in $y$ of the solution 
solves the heat equation on $\T$ and therefore
$$
\|u_0(t) - \int_{\T^2}u\|_{L^2(\T)} \leq e^{-4\pi^2t} \|u_0(0) -
\int_{\T^2}u\|_{L^2(\T)}.
$$
By Proposition \ref{prop:v1} and more precisely \eqref{eq:boundimproved}, 
for $k \neq 0$ we have
\begin{equation}\label{eq:bornefou}
r(0, 2\pi kV) \geq  
\left( \frac{\omega_2(2\pi k V)}{2\pi(1 + {\it Osc}(2\pi k V))}\right)^2 \geq
 \left( \frac{\omega_2(V)}{2\pi(1 + {\it Osc}(V))}\right)^2.
\end{equation}
The conclusion then follows from Parseval-Plancherel identity and
\eqref{eq:weiineq}.
\end{proof}

\medskip

Note that in view of the inequality $\omega_2(V) \leq \frac12 {\it Osc}(V)$ the lower 
bound provided in the right hand side of \eqref{eq:prop1} is always bounded from above 
by $\frac{\pi^2}{72|I|^2]}$, in particular independent of $V$. This, as \eqref{eq:bornefou} 
shows, will not allow to prove a regularizing effect through Fourier decomposition. 
The next result is an extension of \cite{Wei21} Lemma 4.3, which studied the case of 
periodic boundary conditions, to the case of Dirichlet boundary conditions. 
It provides an alternative lower bound for $r(\lambda_1)$ which will be useful for 
studying large Fourier sectors.

\smallskip

For $0 < \eps < |I|/2$, define the quantity 
\begin{equation}\label{def:omega1}
\omega_1(\eps, V) := \inf_{\substack{J \subseteq I\\J \text{ interval}\\|J| \geq
2\eps}} \quad \inf_{p, q \in \R} \left\{
	\int_J |PV(x) - (px+q)|^2\, dx \right\},
\end{equation}
where $PV$ refers to an arbitrary primitive of $V$ on $I$.
Similar to $\omega_2(V)$, it measures how far $V$ is from being constant, but it
introduces an additional scale $\eps$. In particular, $\omega_1(\eps, V) > 0$ if and 
only there is no interval $J\subset I$ of length $2\eps$ on which $V$ is 
a constant. 

\begin{prop}[Extended from \cite{Wei21}]\label{prop:v2}  For arbitrary $0 < \eps < |I| / 2$,
\begin{equation}\label{eq:prop2}
r(\lambda_1, V) + \lambda_1 \geq \frac{1}{\eps^2} \varphi^{-1}\big(\eps
\omega_1(\eps, V)\big)^2
\end{equation}
where $\varphi:[0,\frac{\pi}{2}) \to [0,+\infty)$ is the bijection $\varphi(s) := 36 s
\tan(s)$.
\end{prop}

Note that in the case of Dirichlet boundary conditions $\lambda_1 > 0$ and 
Proposition \ref{prop:v2} only provides a non trivial lower bound for 
$r(\lambda_1)$ if for some value of $\eps$ the r.h.s. in \eqref{eq:prop2} is 
strictly larger than $\lambda_1.$ On the other hand, because of the term
$1/\eps^2$, that same r.h.s. might become large for small values of $\eps$ 
if $V$ allows it. 
This is in contrast with the lower bound for $r(\lambda_1,V)$ given by  
Proposition \ref{prop:v1}, which, as already mentioned, is always positive but on 
the other hand is also bounded above by a constant depending only on $|I|$, in particular 
independent of $V$.

\begin{proof}[Proof of Proposition \ref{prop:v2}]
The argument in the case of periodic boundary conditions is presented in  
details in \cite{Wei21} Section 4. Since it carries over to the Dirichlet case 
with minor adaptations, we only highlight the key steps and differences.
Note also that \eqref{eq:prop2} is immediate if $V$ is identically constant,
because $\omega_1$ vanishes in that case. In the sequel we therefore assume
that $V$ is not identically constant on $I$.

Fix $s \in \R$ and let 
\begin{equation}\label{def:mu}
\mu := \inf\left\{ \|(Au - (\lambda_1 + is) u\|, 
\ u \in \mathcal{D}(A),\ \|u\| = 1\right\}.
\end{equation}
Since $s$ is arbitrary, by definition \eqref{eq:rlambda1} it suffices 
for the proof to show that the equivalent of \eqref{eq:prop2} holds 
with $r(\lambda_1, V)$ replaced by $\mu.$ 

From \cite{Wei21} Lemma 4.1, which applies 
word for word, the infimum defining $\mu$ is achieved, and 
there exist a minimizer $u_0$ satisfying the equation
\begin{equation}\label{eq:minimizer}
   Au_0 -(\lambda_1 + is)u_0 = \mu \bar{u_0}. 
\end{equation}
It follows in particular that $\mu$ is positive. Indeed, by
contradiction if $\mu = 0$ then $u_0$ would need to be proportional to $e_1$, 
and at the same time satisfy the equation
$$
0 = -\partial_{xx} e_1 - \lambda_1 e_1 = -i(V-s)e_1 \qquad\text{on } I,
$$
which obviously doesn't hold because $e_1 >0$ in the interior of $I$ and $V$ is 
not identically constant on $I$. 

In the sequel we denote by $u$ a suitable multiple of $u_0$   so $\|u\|_\infty = 1.$ 
Choose then $x_0 \in I$ s.t. $|u(x_0)| = \|u\|_\infty = 1$, and set 
$x_- = \max\{a \leq x \leq x_0 \ | \ u(x) = 0\}$ and 
$x_+ = \min\{x_0 \leq x \leq b \ | \ u(x) = 0\}$. Here the maximum and minimum can 
indeed be achieved since $u$ is continuous and $u(a)=u(b)=0$. For $x$ in the open 
interval $I_\pm := (x_-,x_+)$ we have $|u(x)|> 0$ and may write 
$u(x) = \exp(\rho(x) + i \theta(x))$, with $\rho$ and $\theta$ being real 
valued. In particular $\rho(x) = \log(|u(x)|)$, $\rho(x)\leq \rho(x_0) = 0$ 
and $\rho'(x_0)=0$.  
Equation $Au - (\lambda_1 + i s)u = \mu \bar{u}$ translates within $I_\pm$ into
\begin{equation}\label{eq:rhotheta}
\left\{
\begin{array}{l}
-\rho''(x) - \rho'(x)^2 + \theta'(x)^2 - \lambda_1 = \mu \cos(2\theta(x)),\\
-\theta''(x) - 2\rho'(x)\theta'(x) + V(x) - s = -\mu \sin(2\theta(x)).
\end{array}
\right.
\end{equation}
Set $\Lambda := \sqrt{\mu + \lambda_1}$ and then the function 
$$
\rho_1(x) := \arctan\big(\frac{\rho'(x)}{\Lambda}\big) \qquad\text{in
}I_\pm.
$$
Observe that 
$$
\frac{\rho_1'}{\Lambda} + 1 = \frac{\rho''}{\Lambda^2 + (\rho')^2} + 1 = \frac{\rho'' +
(\rho')^2 + \Lambda^2}{(\rho')^2 + \Lambda^2} = \frac{(\theta')^2 + 2\mu
\sin^2(\theta)}{(\rho')^2 + \Lambda^2},
$$
where the last identity immediately follows from the first equation in
\eqref{eq:rhotheta}, and therefore, 
\begin{equation}\label{eq:rho1ineq}
\rho_1(x_0) = 0, \text{ and } \rho_1'(x) \geq - \Lambda \quad
\forall\, x \in I_\pm.
\end{equation}
Since $\rho(x_0) = 0$ and $\rho(x) = \log(|u(x)|)$ tends to $-\infty$ at
$x_\pm$, it follows that $\rho'$ is unbounded from below on $[x_0,x_+)$ 
and unbounded from above on $[x_-,x_0),$ and therefore that 
$\inf \{\rho_1(x),\ x \in [x_0,x_+)\} = -\frac{\pi}{2}$ and $\sup \{\rho_1(x),\ x 
\in [x_-,x_0)\} = \frac{\pi}{2}.$
From \eqref{eq:rho1ineq}, this yields
\begin{equation}\label{eq:bornexpm}
\big| x_\pm - x_0 \big| \geq \tfrac{\pi}{2\Lambda},
\end{equation}
as well as, in view of the definition of $\rho_1$, 
\begin{equation}\label{eq:bornerhoprime}
\big|\rho'(x)\big| \leq \Lambda \tan(|x-x_0|\Lambda)\qquad \forall\, x \in
(x_0-\tfrac{\pi}{2\Lambda},x_0+\tfrac{\pi}{2\Lambda}).
\end{equation}
Set $\eps_0 :=  \frac{\pi}{2\Lambda}$. We distinguish two cases for an arbitrary $0 < \eps < |I| / 2$:\\
\noindent{\bf Case 1:} If $\eps_0 \leq \eps < |I|/2$. Then 
\begin{equation}\label{eq:ok1}
\mu + \lambda_1 = \Lambda^2 \geq \frac{\pi^2}{4\eps^2} \geq \frac{1}{\eps^2}
\varphi^{-1}\big(\eps\omega_1(\eps, V)\big)^2,
\end{equation}
simply because $\varphi^{-1}$ is bounded above by $\frac{\pi}{2}.$ 

\noindent{\bf Case 2:} If $0 < \eps < \eps_0.$ Then 
consider the interval $J := [x_0-\eps, x_0 +\eps]$, which is a subset 
of $I_\pm$ by \eqref{eq:bornexpm}. 
From \eqref{eq:bornerhoprime} and the first equation in \eqref{eq:rhotheta} 
we deduce, following the lines in \cite{Wei21} Lemma 4.2, that 
\begin{equation}\label{eq:bornetheta2}
\int_J|\theta'(x)|^2\, dx \leq 
4\Lambda\tan(\eps \Lambda).
\end{equation}
Similarly, integrating the second equation in \eqref{eq:rhotheta} it follows 
that
\begin{equation}\label{eq:PVclose1}
\big| PV(x) - \theta'(x) -(px + q)| \leq 2\sqrt{2}\Lambda\tan(\Lambda|x-x_0|) -
\sqrt{2}\lambda_1 |x-x_0|,
\end{equation} 
where $p := s$ and $q := PV(x_0) - \theta'(x_0) - sx_0,$ and therefore that 
\begin{equation}\label{eq:PVclose2}\begin{split}
\int_J \big| PV(x) - \theta'(x) -(px + q)|^2 \,dx &\leq
\int_J 8\Lambda^2\tan^2(\Lambda|x-x_0|) \, dx\\
&\leq 16\Lambda\tan(\Lambda\eps).
\end{split}\end{equation}
Combining \eqref{eq:bornetheta2} and \eqref{eq:PVclose2} yields
$$
\eps \int_J |PV(x) - (px+q)|^2\,dx \leq 36\eps\Lambda\tan(\Lambda\eps) =
\varphi(\Lambda \eps),
$$
and therefore by definition of $\omega_1$ \eqref{def:omega1} that $\eps \omega_1(\eps, V) \leq
\varphi(\Lambda \eps)$. Since $\varphi$ is monotone increasing, we obtain 
as in \eqref{eq:ok1}
$$
\mu + \lambda_1 = \Lambda^2 \geq \frac{1}{\eps^2} \varphi^{-1}\big(\eps
\omega_1(\eps, V)\big)^2.
$$
This completes the proof of Proposition \ref{prop:v2}.
\end{proof}

\section{Weak local H\"ormander condition}\label{sect:Hor}

The purpose of this section is to present the proof of Theorem \ref{thm:Hor}.

\medskip

Without loss of generality, we may assume for the exposition that the ``good interval'' $I$ 
in Assumption $(H)$ is given by $I = [0,|I|]$.  
Also, because \eqref{eq:lowerbH} will eventually show that the lower bound is 
independent of $(x,y) \in \T^2,$ we shall use the notation 
$\mu(t) := P_t(x,y,\cdot)$ to stress that the results obtained are independent 
of $(x,y) \in \T^2.$ For $t > 0$, we denote by 
$$
G_t(x) := \frac{1}{\sqrt{4\pi t}} e^{-\frac{x^2}{4t}}
$$
the fundamental solution of the 1D heat equation on $\R$, and by 
$G_t^\T$ (resp. $G_t^I$) the corresponding fundamental solutions on
$\T$ (resp. on $I$ with homogeneous Dirichlet conditions).
By comparison principles, for all $t>0$ we have
\begin{equation}\label{eq:compgreen}
G_t^\T(x,x') \geq G_t(x-x') \ \forall\, x,x' \in \T, \qquad
G_t^I(x,x') \leq G_t(x-x') \ \forall\, x,x' \in I.
\end{equation}

\medskip

The first step of the argument is to notice that the probability 
measures $\int_{\T} \mu(t)\, dy$ defined on $\T$ by 
$$
\left(\int_{\T}\mu(t)\, dy \right)(B) := \mu(t)(B \times \T) \qquad \forall\, B \in \mathcal{B}(\T)
$$
are solution to the heat equation on $\T$, and therefore from \eqref{eq:compgreen} 
and as in  \eqref{eq:borneGtore} we deduce that 
\begin{equation*}\label{eq:borneGraison} 
\int_{\T}\mu(\frac{1}{8}) \, dy \geq \kappa_0 := \inf_{x,x' \in \T} 
G_\frac18^\T(x,x') \geq G_\frac18(\frac 12) = \sqrt{\frac{2}{e\pi}}.  
\end{equation*}
Since the heat equation preserves positivity, it 
follows that 
\begin{equation}\label{eq:lowerbheat}
\int_{\T}\mu(t) \, dy \geq \kappa_0, \qquad \forall\,
t \geq \frac18.
\end{equation}
Returning to the transition probabilities notations, \eqref{eq:lowerbheat} 
simply translates into 
\begin{equation*}
P_t(x,y,B \times \T) \geq \kappa_0 |B|, \qquad \forall\,
t \geq \frac18,
\end{equation*}
for any  $(x,y) \in \T^2$ and any $B \in \mathcal{B}(\T)$.

We now consider the equivalent to equation \eqref{eq:u} but set on the 
domain $I \times \T$ with homogeneous Dirichlet boundary conditions
on $\partial I \times \T.$ In terms $(X_t, Y_t)$, this amounts 
to kill the process whenever $X_t$ exits the interval $I.$ We denote by
$(\nu(t))_{t \geq 0}$ the solution corresponding to the initial
data\footnote{Here and in the sequel we use the symbol $\rest$ to denote the 
restriction of a function or a measure to a smaller domain.} 
$\nu(0) = \mu(\frac18) \rest_{I \times \T}.$ By the comparison
principle, it holds
\begin{equation}\label{eq:comparemunu}
\mu(t + \frac18) \geq \nu(t) \quad\text{on } I \times \T,\quad  \forall\, t \geq 0.
\end{equation}
We decompose in Fourier series 
$$
\nu(t) =: \sum_{k \in \Z} \nu_k(t) e^{2i\pi ky},
$$
where the measures $\nu_k(t) \in \mathcal{M}(I, \C)$ are then 
(weak) solutions to the equations 
\begin{equation}\label{eq:fouriernu}
\partial_t \nu_k - \partial_{xx} \nu_k + 2i\pi kV(x)\nu_k = 0
\end{equation}
on $\R_+ \times I$ with homogeneous Dirichlet boundary 
conditions on $\partial I$.

For $k = 0$, it follows from \eqref{eq:lowerbheat} that 
\begin{equation}\label{eq:lowerbnu0}
\nu_0(0) = \left(\int_{\T} \mu(\frac18)\, dy\right) \rest_{I} 
\geq \kappa_0 \geq \kappa_0 \sin(\frac{\pi}{|I|}x) \quad \text{on } I,
\end{equation}
and therefore that
\begin{equation}\label{eq:lowerbnut}
\nu_0(t) \geq \kappa_0 e^{-\frac{\pi^2}{|I|^2}t}
\sin(\frac{\pi}{|I|}x) \quad \text{on } I, \quad \forall\, t \geq 0.
\end{equation}
We claim that, under assumption $(H)$, at time $T := t_H - \frac14$ and for $\kappa_1 :=
\kappa_0/(2\sqrt{2})$ we have  
\begin{equation}\label{eq:nukperturb}
\sum_{k \neq 0} \|\nu_k(T)\|_{L^\infty(I)} \leq \kappa_1 e^{-\frac{\pi^2}{|I|^2}T},
\end{equation}
so that in particular from \eqref{eq:lowerbnut}
\begin{equation}\label{eq:nukperturb2}
\nu(T) \geq \nu_0(T) - \sum_{k \neq 0} \|\nu_k(T)\|_{L^\infty(I)}
\geq \kappa_1e^{-\frac{\pi^2}{|I|^2}T} \qquad\text{on } I'\times \T,
\end{equation}
where here and in the sequel $I' := \{x \in I \text{ s.t. } d(x,I^c) \geq
|I|/4\}.$ Postponing the proof of \eqref{eq:nukperturb}, we then deduce from 
 \eqref{eq:comparemunu} and \eqref{eq:nukperturb2} that
\begin{equation}\label{eq:bornelocalemu2}
\mu(t_H - \frac18) \geq \kappa_1e^{-\frac{\pi^2}{|I|^2}T} \qquad \text{on } I'\times \T.
\end{equation}
To globalize the previous lower bound to the whole of $\T^2$, we simply
use the Chapman-Kolmogorov equation and write (similar to the last step in Section~\ref{sect:plateau})
$$
\mu(t_H)(B) = \int \mu(t_H - \frac18)(dx'dy') P_\frac18(x',y',B) \qquad \forall\,
B \in \mathcal{B}(\T^2). 
$$
Using once more the fact that averaging of \eqref{eq:u} in $y$ yields the 
1D heat equation on $\T$, combined with \eqref{eq:bornelocalemu2}, we therefore obtain
\begin{equation*}\begin{split}
P_{t_H}(x,y, B) = \mu(t_H)(B) &\geq \kappa_1 e^{-\frac{\pi^2}{|I|^2}T} \int_{I'\times \T}
P_\frac18(x',y', B) \, dx'dy'\\
&= \kappa_1 e^{-\frac{\pi^2}{|I|^2}T} \int_{\T^2}\int_{I'}
G_\frac18^\T(x',x^*) \, dx'\: 1_{B}(x^*,y^*)\, dx^*dy^*\\
&\geq \frac{\kappa_1}{2}|I| e^{-\frac{\pi^2}{|I|^2}T}\Big(\inf_{x',x^* \in
\T}G_\frac18^\T(x',x^*)\Big) |B|\\
&= \frac{\kappa_0\kappa_1}{2}|I|e^{-\frac{\pi^2}{|I|^2}T} |B| \geq
\frac{1}{3}|I| e^{-\frac{\pi^2}{|I|^2}t_H} |B|,
\end{split}
\end{equation*}
for any $B \in \mathcal{B}(\T^2)$, which is nothing but \eqref{eq:lowerbH}. 

\smallskip

It remains to prove our claim \eqref{eq:nukperturb}.\\
Remark first that since $\nu_k(t) = \int \nu(t) e^{-2i\pi ky}dy$, for any 
$t\geq 0$ we have 
\begin{equation}\label{eq:firstboundnu}
\|\nu_k(t)\|_{\mathcal{M}(I)} \leq \|\nu(t)\|_{\mathcal{M}(I \times \T)} \leq
\|\mu(t)\|_{\mathcal{M}(\T^2)} = 1.  
\end{equation}

For $k \neq 0$ fixed, we will use the regularizing effect of
\eqref{eq:fouriernu}: first from $\mathcal{M}(I)$ into $L^2(I)$ 
(with a possible polynomial growth in $|k|$) by viewing the non-homogeneous
term in \eqref{eq:fouriernu} as a forcing, then the long time
and $k$-dependent exponential decay in $L^2(I)$ provided by the 
analysis in Section \ref{sect:spectral}, and finally from $L^2(I)$ 
into $L^\infty(I)$ (still with a possible polynomial growth in $|k|$). 
For that purpose, consider a solution of the equation 
$\partial_t f - \partial_{xx} f = g$ on $(t,x) \in U \times U$  where $U :=
[0,1]$, with homogeneous Dirichlet conditions on $U \times \partial U$. Whenever
$f(0) \in \mathcal{M}(U)$ and $g \in L^\infty(U, \mathcal{M}(U))$, 
we have
\begin{equation}\label{eq:regu1}\begin{split}
\|f(1)\|_{L^2(U)} &\leq \|e^{-t\partial_{xx}}f(0)\|_{L^2(U)} + \int_0^1 
\|e^{-(1-s)\partial_{xx}}g(s)\|_{L^2(U)}\, ds\\
&\leq \|G_1\|_{L^2(\R)}\|f(0)\|_{\mathcal{M}(U)} + \int_0^1
\|G_{1-s}\|_{L^2(\R)}\|g(s)\|_{\mathcal{M}(U)}\, ds\\
&\leq \frac{1}{\sqrt{8\pi }}\Big( \|f(0)\|_{\mathcal{M}(U)} + 2 \sup_{t \in
U}\|g(t)\|_{\mathcal{M}(U)}\Big).
\end{split}\end{equation}
Similarly, if $f(0) \in L^2(U)$ and $g \in L^\infty(U, L^2(U))$
we have
\begin{equation}\label{eq:regu2}\begin{split}
\|f(1)\|_{L^\infty(U)} &\leq \|e^{-t\partial_{xx}}f(0)\|_{L^\infty(U)} + \int_0^1 
\|e^{-(1-s)\partial_{xx}}g(s)\|_{L^\infty(U)}\, ds\\
&\leq \|G_1\|_{L^2(\R)}\|f(0)\|_{L^2(U)} + \int_0^1
\|G_{1-s}\|_{L^2(\R)}\|g(s)\|_{L^2(U)}\, ds\\
&\leq \frac{1}{\sqrt{8\pi }}\Big( \|f(0)\|_{L^2(U)} + 2 \sup_{t \in
U}\|g(t)\|_{L^2(U)}\Big).
\end{split}\end{equation}
Applying \eqref{eq:regu1} to \eqref{eq:fouriernu}
we obtain, after a straightforward parabolic scaling, a constant shift in $V$, and taking into account
the uniform bound \eqref{eq:firstboundnu} as well as the fact (to globalize in
$t$) that \eqref{eq:fouriernu} is a contraction in $L^2(I)$, that 
\begin{equation}\label{eq:regunuk1}
\|\nu_k(t)\|_{L^2(I)} \leq \frac{1}{\sqrt{8\pi|I| }}\Big(1 +
2\pi|k|{\it Osc}(V)|I|^2\Big),
\end{equation}
for all $t \geq |I|^2.$ Similarly, from \eqref{eq:regu2} we obtain 
\begin{equation}\label{eq:regunuk2}
\|\nu_k(t + |I|^2)\|_{L^\infty(I)} \leq \frac{1}{\sqrt{8\pi|I| }}\Big(1 +
2\pi|k|{\it Osc}(V)|I|^2\Big)\|\nu_k(t)\|_{L^2(I)},
\end{equation}
for all $t > 0.$

It remains to analyze the long time exponential decay of the $L^2$ norm using
Proposition \ref{prop:v1} and/or Proposition \ref{prop:v2}. Note that in doing
so, the $V$ involved in the skew-adjoint part of what was called operator 
$A$ in Section \ref{sect:spectral} is actually here equal to $2\pi k V$. We
distinguish two cases.

\noindent{\bf Case 1:} We assume that
\begin{equation}\label{eq:grandk}
k^2 \geq \frac{9}{2\pi}e^\frac{8K}{|I|^2}.
\end{equation}
We first rephrase assumption $(H)$ as 
\begin{equation}\label{eq:omega1borne}
\omega_1(\eps, V) \geq \frac{1}{2\eps}e^{-\frac{K}{4\eps^2}} \qquad\forall\, 0 <
\eps \leq \frac{|I|}{2},
\end{equation}
where the function $\omega_1$ was defined in \eqref{def:omega1}. Since $\omega_1$ 
is continuous and increasing in $\eps$, and vanishes as $\eps \to 0$, from 
\eqref{eq:grandk} we may find $0< \eps(k) \leq \frac{|I|}{2}$ such that 
$$
\eps(k)4\pi^2k^2\omega_1(\eps(k), V) = 9\pi,
$$
and from \eqref{eq:omega1borne} we deduce the lower bound
$$
\frac{1}{\eps(k)^2} \geq \frac{4}{K}\log\Big(\frac{2\pi}{9}k^2\Big). 
$$
Proposition \ref{prop:v2} together with the identity $\varphi(\frac{\pi}{4}) =
9\pi$ then yields 
\begin{equation*}\label{eq:borner1}\begin{split}
r(\lambda_1, 2\pi k V) &\geq \frac{\pi^2}{4K}\log\Big(\frac{2\pi}{9}k^2\Big) -
\lambda_1\\
&\geq \frac{\pi^2}{8K}\log\Big(\frac{2\pi}{9}k^2\Big),
\end{split}\end{equation*}
where we have used once more \eqref{eq:grandk}, and also the identity $\lambda_1 =
\frac{\pi^2}{|I|^2}.$ Combining the previous inequality with
\eqref{eq:weiineq}, \eqref{eq:regunuk1} and \eqref{eq:regunuk2}, we finally deduce that
$$
\|\nu_k(t + 2|I|^2)\|_{L^\infty(I)} \leq \frac{\beta^2 k^2}{8\pi|I|}
e^{\frac{\pi}{2}} \Big(\frac{9}{2\pi k^2}\Big)^{\frac{\pi^2}{8K}t} e^{-\lambda_1
t}, \quad \forall\, t > 0
$$
(recall $\beta = 1 + 2\pi{\it Osc}(V)|I|^2$ as in the statement of Theorem \ref{thm:Hor}), which 
we may rewrite as
$$
\|\nu_k(t)\|_{L^\infty(I)} \leq \frac{9\beta^2}{16\pi^2|I|}
e^{\frac{\pi}{2} + 2\pi^2} \Big(\frac{9}{2\pi
k^2}\Big)^{\frac{\pi^2}{8K}(t-2|I|^2) - 1} e^{-\lambda_1 t}, \quad \forall\, t
\geq 2|I|^2.
$$
In view of our goal \eqref{eq:nukperturb}, we notice that for $t = T = t_H -
\frac14$, the exponent $\frac{\pi^2}{8K}(t-2|I|^2) - 1$ is larger than one, 
and since besides we have $2\pi k^2 \geq 9$, it follows that
\begin{equation*}
\|\nu_k(T)\|_{L^\infty(I)} \leq \frac{81\beta^2}{32\pi^3|I|}
e^{\frac{\pi}{2} + 2\pi^2} \frac{1}{k^2} e^{-\lambda_1 T},
\end{equation*}
and hence that for any $\Lambda \in \N$ such that $\Lambda^2 \geq
\frac{9}{2\pi}e^{\frac{8K}{|I|^2}},$
\begin{equation*}
\sum_{|k| > \Lambda}\|\nu_k(T)\|_{L^\infty(I)} \leq \frac{81\beta^2}{16\pi^3|I|}
e^{\frac{\pi}{2} + 2\pi^2} \frac{1}{\Lambda} e^{-\lambda_1 T}.
\end{equation*}
In view of our goal \eqref{eq:nukperturb} once more, we now fix  
\begin{equation}\label{def:Lambda}
\Lambda := \max\Big(\frac{3}{\sqrt{2\pi}}e^{\frac{4K}{|I|^2}},
\frac{81\beta^2}{4\pi^{5/2}|I|}e^{1 + \frac{\pi}{2} + 2\pi^2} \Big)
\end{equation}
and by construction we so obtain  
\begin{equation}\label{eq:finalcase1}
\sum_{|k| > \Lambda}\|\nu_k(T)\|_{L^\infty(I)} 
\leq \frac{1}{4\sqrt{\pi e}} e^{-\frac{\pi^2}{|I|^2} T} 
= \frac{\kappa_1}{2} e^{-\frac{\pi^2}{|I|^2} T}.
\end{equation}

\noindent{\bf Case 2:} We assume that $0 \neq |k| \leq \Lambda.$

In that situation we rely on Proposition \ref{prop:v1} instead, providing the
($k$ independent) lower bound 
$$
r(\lambda_1, 2\pi kV) 
\geq \frac{4\pi^2k^2
\omega_2(V)^2}{18\big(\frac{\pi^2}{|I|^2} + 4 k^2{\it Osc}(V)^2|I|^2\big)}
\geq \frac{2\omega_2(V)^2|I|^2}{9\beta^2},
$$
and then, along the same lines as above, the estimate
\begin{equation}\label{eq:finalcase2}
\sum_{|k| \leq \Lambda}\|\nu_k(T)\|_{L^\infty(I)} \leq D  e^{-\lambda_1 T},
\end{equation}
where
$$
D := \frac{\beta^2(\Lambda + 1)^3}{12\pi|I|} e^{\frac{\pi}{2} + 2\pi^2}
e^{-\frac{2\omega_2(V)^2|I|^2}{9\beta^2})(T-2|I|^2)}.
$$
If we can show that 
\begin{equation}\label{eq:borneD}
D \leq \frac{1}{4\sqrt{e\pi}} = \frac{\kappa_1}{2},
\end{equation}
then the claim \eqref{eq:nukperturb} will follow from summation of
\eqref{eq:finalcase1} and \eqref{eq:finalcase2}. Inverting the relation defining
$D$, and recalling $T = t_H - \frac{1}{4}$, \eqref{eq:borneD} is seen equivalent to 
\begin{equation}\label{eq:borneH}
t_H \geq \frac14 + 2|I|^2 + \frac{9\beta^2}{2|I|^2\omega_2(V)^2} \left[ \log\Big(
\frac{\beta^2(\Lambda + 1)^3}{3 \sqrt{\pi}}\Big) + \frac12 + \frac{\pi}{2} +
2\pi^2\right].
\end{equation}
At this point it suffices to recall the definition \eqref{def:Lambda} 
of $\Lambda$ to realize that \eqref{eq:borneH} is satisfied for our 
defining choice of $t_H$ in \eqref{def:tH}. As a matter of fact, the definition
\eqref{def:tH} is only a cosmetic improvement to the expression on the r.h.s.  
of \eqref{eq:borneH}, the details, which are omitted, only rely on explicit 
inequalities on universal constants to reduce their number, and the fact that 
$|I| \leq 1$ and $K \geq 1.$ This completes the proof of Theorem \ref{thm:Hor}.\qed

\section{Additional remarks}\label{sec:fourretout}

\subsection{On the lack of smoothing}
As we mentioned in the introduction, when $V$ is not hypoelliptic the smoothing
properties of Equation~\eqref{eq:u} may be very limited. Smoothing in $x$ of
course always occur, although it is in general limited to $W^{2,\infty}$ when
$V$ is barely bounded measurable, but the key point is the smoothing in $y$. 

When $V$ has a plateau, then it is easy to realize that the fundamental solution 
corresponding to an initial Dirac delta located within the plateau, when
integrated in $x$, will conserve, for all positive times, an atomic part. In a
probabilistic view-point, this is related to the well-known fact that given a  
fixed open interval $I$ around their initial position and a fixed end time $T$,  
Brownian particles remain in $I$ at least up to time $T$ with a positive
probability. 

The next example was computed by L\'evy \cite{Levy39} in 1939 (see also e.g.
\cite{Ber21} for a more recent account), and is one of the
few processes with explicitly computable laws.  Consider a Wiener process 
$(\mathcal{W}_t)_{t\geq 0}$ in the real line starting at $x = 0$ and for $t > 0$ 
define the average time spent in the positive half-space:
$$
Z_t := \frac{1}{t}\int_0^t 1_{\mathcal{W}_s > 0}\, ds.
$$
Then the law of the process $Z_t$ is the so-called arcsine law, i.e. 
\begin{equation}\label{eq:arcsine}
	P(Z_t < a) = \frac{2}{\pi} \arcsin\big(\sqrt{a}\big)\qquad \forall\, 0 \leq a \leq
	1.
\end{equation}
Consider then the equation
$$
\partial_t u - \partial_{xx} u + H(x)\partial_y u = 0
$$
on $\R_+ \times \R^2.$ Then in view of \eqref{eq:sp} the marginal in $y$ of
$S(t)\delta_{(0,0)}$ has a law given by $Y_t = \sqrt{2} t Z_t$, and therefore
$$
P(Y_t < b) = \frac{2}{\pi}\arcsin\big(\sqrt{\tfrac{b}{\sqrt{2}t}}\big) \qquad \forall\, 0 < b < \sqrt{2}t,
$$
so that after differentiation
$$
\int_\R u(t,x,y)\, dx = \tfrac{1}{\pi
\sqrt{y(\sqrt{2}t-y)}} \notin L^2_{loc}(dy).
$$
Although that example is set on $\R^2$ instead of $\T^2$, that difference 
shouldn't have much influence on the short time behavior of solutions with 
an initial Dirac delta located at a jump point of $V(x)$, and strongly suggest
that no $L^1$ to $L^2$ smoothing may be expected in such situations, even
letting aside plateaux. 

\subsection{On existing alternative strategies}
Starting around the beginning of this century, a number of works have flourished 
with the general idea that hypoellipticity assumptions could be relaxed, as long 
as convergence to equilibrium is the main objective rather than regularity, to 
weaker forms that focus on identifying Lyapunov type functionals adapted to the 
framework at hand.  These methods have been associated with the name of 
{\it hypocoercivity} to stress that change of focus, and in many cases they 
amount to construct a perturbed norm, through the addition of ``mixed'' terms 
that play a role similar to commutators in hypoelliptic frameworks, that is then 
shown to decay in time. Such methods  have the advantages that in principle they 
should be more robust to extensions to nonlinear models and also do not 
necessarily require a  regularizing effect.

In \cite{Vil09}, Villani proposed a number of abstract Hilbert space frameworks 
that imply hypocoercivity, and used them to (re/im)prove existing relaxation 
results both for linear models, such as the kinetic Fokker-Planck equation, 
and nonlinear ones, in particular the Boltzmann equation. For the latter, and 
due to the notably weak regularizing properties of collision operators, the 
convergence was shown to be faster than any negative power but not exponential. 
For a linear Boltzmann type model, Hérau \cite{Her06} was able to prove 
exponential convergence using non local mollifiers in the construction of the 
Lyapunov functional. This idea was generalized to an abstract Hilbert framework 
by Dolbeault, Mouhot and Schmeiser \cite{DolMouSch15}, and later used in different 
contexts. In a different direction, and with the goal of extending the 
applicability of the method to cases where the natural norm is not hilbertian, 
such as $L^1$ or measure spaces, Gualdani, Mischler and Mouhot \cite{GuaMisMou17} 
presented an abstract framework allowing to extend semi-group decay estimates 
from a smaller Banach (in particular Hilbert) space, to a larger Banach space. 

\smallskip

Below, we briefly report on our naive tests at using the abstract frameworks 
of \cite{Vil09, DolMouSch15, GuaMisMou17} for Equation~\eqref{eq:u}. In short, 
although they require less regularity than hypoellipticity, each commutator 
is ``burning'' one derivative of $V$ and that prevented us from obtaining 
the equivalent of what spectral or probabilistic methods yielded in previous 
sections. This limitation is not seen in models originating from kinetic theory, 
since in these cases $V(x)=x$ ($x=v$ is the velocity variable). For ease of 
reference, we use the notations in the respective papers below.  

First, regarding Theorem 24 in \cite{Vil09}, with $A = \partial_x$ and 
$B = \partial_t + V(x)\partial_y$. If $V$ is regular and does not suffer from 
degeneracy, in particular if it doesn't have plateaux, it seems natural to 
choose $N_c = 1$, so that $C_2 = 0$ and (with the naive choice of $Z_1 = I$ and 
$R_1 = 0$) $C_1 = [A,B] = V'(x)\partial_y$ and $R_2 = [C_1, B] = 0$.  
Assumption $i)$ then requires that $[A, C_1] = V''(x)\partial_y$ should be 
relatively bounded to $C_1 = V'(x)\partial_y$ and $A^2 = \partial_{xx}$, which 
amounts to a regularity assumption on $V$. The same applies to Theorem 28.

Next, concerning Theorem 2 in \cite{DolMouSch15}. Here the corresponding setting 
is $L = \partial_{xx}$, $T = V(x)\partial_y$, $\Pi f(y) = \int_{\T}f(x,y)\, dx$ 
takes the mean in $x$, and only functions that have global mean equal to zero 
are considered. The ``microscopic'' coercivity assumption $(H_1)$ is simply 
the Poincaré-Wirtinger inequality in $x$, integrated in $y$, and the ``macroscopic'' 
coercivity assumption $(H_2)$ holds with $\lambda_m = \int_\T V^2(x)\, dx.$ 
Assumption $(H_3)$ about the equality $\Pi T \Pi = 0$ is verified as well, 
but not assumption $(H_4)$ since the mollifier $A$ does not cope for the 
loss of regularity in $y$ unless additional assumptions on $V$ are required.

Finally, regarding the framework in \cite{GuaMisMou17}, it might require less 
regularity on $V$ but as we just saw in the previous subsection plateaux of 
$V$ prevent the $L^1$ to $L^2$ smoothing of the semi-group, a situation 
which seems to be a pre-requisite for the approach.

\subsection{On an optimal control approach}\label{sect:hj}
The following approach to mixing is partly inspired from the analysis of Li and
Yau \cite{LiYau86} of the heat kernel for Schr\"odinger operators, and especially
Section 2 related to Harnack inequalities. It provides in particular an elegant 
construction of the fundamental solution to the Fokker Planck equation, which
was explicitly computed by Kolmogorov \cite{Kol34} in 1934, without
indication of the method\footnote{Presumably Fourier transform.}. 

\medskip

Because Equation~\eqref{eq:u} is positivity preserving, it is natural to 
perform the change of unknown 
\begin{equation}\label{def:psi}
u = e^{-\psi},
\end{equation}
which then yields the equation
\begin{equation}\label{eq:psi}
\partial_t \psi + |\partial_x \psi|^2 - \partial_{xx}\psi +
V(x)\partial_y \psi = 0
\end{equation}
for $\psi.$ Let us momentarily ignore the second order term in \eqref{eq:psi}, 
and consider the corresponding nonlinear first order equation 
\begin{equation}\label{eq:psibis}
\partial_t \psi + |\partial_x \psi|^2 +
V(x)\partial_y \psi = 0.
\end{equation}
It is a well-known fact (see e.g. \cite{PonBoGaMi62}) that, at least formally, 
Hamilton-Jacobi equations written under the form
\begin{equation}\label{eq:HJ}
\left\{
\begin{array}{l}
\partial_t \psi + H(z, \nabla \psi) = 0,\\
H(z, p) = \max_w p\cdot f(z,w) - g(z, w),
\end{array}\right.
\end{equation}
are related to the optimal control problems
\begin{equation}\label{eq:control}
\left\{
\begin{array}{l}
\psi(t,z) := \inf \int_0^t g(Z(s), w(s))\, ds,\\
\text{with } \dot{Z}(s) = f(Z(s), w(s)),\ Z(0) = z_0,\ Z(t) = z,
\end{array}\right.
\end{equation}
where the counter part in \eqref{eq:HJ} for the initial trajectory position $z_0$ of 
\eqref{eq:control} is the initial cost
\begin{equation*}
\psi(t=0,z)=\left\{
\begin{array}{ll}
0 & z=z_0,\\
+\infty& \text{otherwise,}
\end{array}\right.
\end{equation*} and the corresponding initial Dirac delta in terms of $u$.

In the special case of Equation~\eqref{eq:psibis}, and writing $z_0 = (x_0,
y_0)$, the previous analogy leads to the
ODE Cauchy problem constraint  
\begin{equation}\label{eq:ODE}
	\left\{\begin{array}{lll}
		\dot{X}(s) = w(s), & X(0) = x_0, &X(t) = x\\
		\dot{Y}(s) = V(X(s)), & Y(0) = y_0, & Y(t) = y,
	\end{array}\right.  
\end{equation}
which of course is reminiscent of \eqref{eq:sp}, and to the cost function 
\begin{equation}\label{eq:cost}
\psi(t,x,y) = \inf_w \int_0^t \frac{1}{4} w^2(s) \, ds.
\end{equation}
Solutions to that optimal control problem can be explicitly constructed in the case 
of the Kolmogorov equation, i.e. when $V(x) = x$ and
the spatial domain is not the torus but the whole space. Indeed, the first order
optimality condition for trajectories then reads
$$
\int_0^t w(s)\varphi(s)\, ds = 0 \qquad\forall\, \varphi \text{ s.t. }
\int_0^t\varphi = \int_0^t\int_0^s \varphi = 0,
$$
and its solutions are affine functions. The end-points constraints turn into 
$$
\begin{array}{l}
	\int_0^t \dot{X}(s) \, ds = \int_0^t w(s)\, ds = x - x_0,\\[2pt]
	\int_0^t \dot{Y}(s) \, ds = \int_0^t X(s)\, ds = \int_0^t [x_0 +
	\int_0^s w(\tau)\, d\tau] ds = y - y_0,
\end{array}
$$
and are easily solved, writing $w(s) = a + 2bs$, as
$$
\begin{pmatrix} a\\b\end{pmatrix} = -\frac{6}{t^3} 
\begin{pmatrix} \frac{t^3}{3} & -t\\ - \frac{t}{2} & 1 \end{pmatrix} 
\begin{pmatrix} x-x_0\\ y - y_0 - x_0t \end{pmatrix}.
$$
Injection into \eqref{eq:cost} then leads to the expression
\begin{equation}\label{eq:explicitKolmo1}
	\psi(t,x,y) = \frac{1}{4t}\big(x-x_0\big)^2 + \frac{3}{t^3}\big( y - y_0 - \tfrac{x +
x_0}{2}t\big)^2.
\end{equation}
At this point we remind that we have initially omitted the second order term
$\partial_{xx} \psi$ when passing from \eqref{eq:psi} to \eqref{eq:psibis}. Note
however that in the special case \eqref{eq:explicitKolmo1} we have
$$
\partial_{xx}\psi = \frac{2}{t} \quad\text{ is constant in space,} 
$$
and therefore replacing $\psi$ by $\psi + 2\log t$ we obtain 
\begin{equation}\label{eq:fundaKolmo}
	u(t,x,y) = \frac{1}{t^2} \exp\Big(\frac{1}{4t}\big(x-x_0\big)^2 + \frac{3}{t^3}\big( y - y_0 - \tfrac{x +
x_0}{2}t\big)^2\Big),
\end{equation}
which, up to a multiplicative normalization factor $\frac{2\sqrt{3}}{\pi}$, is
exactly the expression found in \cite{Kol34} for the fundamental solutions
of
$$
\partial_t u - \partial_{xx}u + x\partial_y u = 0
$$
on $\R_+\times \R^2.$

\medskip

When $V(x)$ is arbitrary, there are of course no explicit expressions for the
optimal trajectories of the control problem, and optimal trajectories may even
not exist in general. Besides, there is absolutely no reason for which
$\partial_{xx} \psi$ would remain constant in space, thus preventing the easy
short-cut above between \eqref{eq:psibis} and \eqref{eq:psi}.

However, eventually we are only interested in pointwise lower bounds for $u$,
that is pointwise upper bounds for $\psi$, and inspection of the argument above 
shows that what eventually matters are upper bounds of the form 
$$
\partial_{xx} \psi \leq C(t),
$$
or more generally convexity upper bounds. In view of its definition as the
infimum of an optimal control problem, in order to estimate second order
variations such as 
$$
\psi(t, x+\eps, y) - 2\psi(t,x, y) + \psi(t,x-\eps, y),
$$
it suffices, $\psi(t,x,y)$ being fixed and given, to exhibit {\it some} admissible
control trajectories for the other two terms which lead to an upper bound in $C(t)\eps^2.$ This in
principle is much more amenable than computing an exact solution, and may
require only minimal assumption on $V$. Yet, there remain difficulties related
to the so-called anormal control trajectories, which arise in particular in the
boundary of the controlled region when e.g. $V$ is bounded. We did not succeed
to fully circumvent them, and therefore leave the question open for further 
investigations.

\medskip

\noindent{ \bf Acknowledgments} 
This work was supported by the  Fondation Simone et Cino Del Duca, Institut de France, 
and by the National Key R\&D Program of China, Project Number 2021YFA1001200. 
Authors also benefited from fruitful discussions with many colleagues including 
Nicolas Fournier, Bernard Helffer, Benoît Perthame and Zhifei Zhang.

\bibliographystyle{abbrv}
\bibliography{Mixing}

{\footnotesize 
\medskip\noindent
{\bf Xu'an Dou}\    
Beijing International Center for Mathematical Research\\
Peking University, China.\\
Email:\: {\tt dxa@pku.edu.cn}

\medskip\noindent
{\bf Delphine Salort}\ Laboratoire Jacques-Louis Lions UMR7598\\
Sorbonne Universit\'e, CNRS, Universit\'e de Paris, Inria, France.\\
Email:\: {\tt delphine.salort@sorbonne-universite.fr}

\medskip\noindent
{\bf Didier Smets}\ Laboratoire Jacques-Louis Lions UMR7598\\
Sorbonne Universit\'e, CNRS, Universit\'e de Paris, Inria, France.\\
Email:\: {\tt didier.smets@sorbonne-universite.fr}

}
\end{document}